\documentclass{article}[12pt]
\usepackage{hyperref}
\usepackage{amsmath,amssymb,latexsym,amsthm}
\usepackage{url}
\usepackage{graphicx,color}
\usepackage{mathdots}
\usepackage[margin=1.2in]{geometry}
\usepackage{caption}
\usepackage{mathtools}
\usepackage{authblk}

\newtheorem{thm}{Theorem}
\newtheorem{lemma}[thm]{Lemma}
\newtheorem{prop}[thm]{Proposition}

\newtheorem{coro}[thm]{Corollary}
\newtheorem{remark}[thm]{Remark}

\DeclarePairedDelimiter\norm{\lVert}{\rVert}

\def\supp{\text{supp}\,}

\def\p{\partial}

\graphicspath{{figures/}}

\def\R{\mathbb R}

\title{A Linearized
Boundary Control Method for the Acoustic Inverse Boundary Value Problem}

\author[1]{Lauri Oksanen \thanks{lauri.oksanen@helsinki.fi}}
\author[2]{Tianyu Yang
\thanks{yangti27@msu.edu}}
\author[3]{Yang Yang
\thanks{yangy5@msu.edu}}
\affil[1]{Department of Mathematics and Statistics, University of Helsinki}
\affil[2]{Department of Computational Mathematics, Science and Engineering, Michigan State University}
\affil[3]{Department of Computational Mathematics, Science and Engineering, Michigan State University}

\date{}

\begin{document}

\maketitle

\abstract{We develop a linearized boundary control method for the inverse boundary value problem of determining a potential in the acoustic wave equation from the Neumann-to-Dirichlet map.
When the linearization is at the zero potential, we derive a reconstruction formula based on the boundary control method and prove that it is of Lipschitz-type stability.
When the linearization is at a nonzero potential, we prove that the problem is of H\"{o}lder-type stability. 
The proposed reconstruction formula is implemented and evaluated using several numerical experiments to validate its feasibility.
}

\section{Introduction}

The paper is concerned with the linearized inverse boundary value problem (IBVP) for the acoustic wave equation with a potential. The goal is to derive stability estimates and reconstruction procedures to numerically compute a small perturbation of a known underlying potential, given the Neumann-to-Dirichlet map on the boundary of the domain.

\textbf{Formulation.}
Specifically, let $T>0$ be a constant and $\Omega\subset\mathbb{R}^n$ be a bounded open subset with smooth boundary $\partial\Omega$. 
Consider the following boundary value problem for the acoustic wave equation with potential: 
\begin{equation} \label{eq:bvp}
\left\{
\begin{array}{rcl}
\Box_{c,q} u(t,x) & = & 0, \quad\quad\quad \text{ in } (0,2T) \times \Omega \\
 \partial_\nu u  & = & f, \quad\quad\quad \text{ on } (0,2T) \times \partial\Omega \\ 
 u(0,x) = \partial_t u(0,x) & = & 0 \quad\quad\quad\quad x \in \Omega.
\end{array}
\right.
\end{equation}
Here $\Box_{c,q}$ is a linear wave operator defined as
$$
\Box_{c,q} u (t,x) :=    \partial^2_t u(t,x) - c^2(x) \Delta u(t,x) + q(x) u(t,x);
$$
$c(x) \in C^\infty(\overline{\Omega})$ is a smooth wave speed that is strictly positive, $q(x)\in L^\infty(\Omega)$ is a real-valued function. We denote this solution by $u^f(t,x)$ when it is necessary to specify the Neumann data.

Given $f\in C^\infty_c((0,2T)\times\partial\Omega)$, the well-posedness of this problem is ensured by the standard theory for second order hyperbolic partial differential equations~\cite{evans1998partial}.
As a result, the following Neumann-to-Dirichlet (ND) map is well defined:
\begin{equation} \label{eq:NDmap}
\Lambda_{c,q} f := u^f|_{(0,2T)\times \partial\Omega}.
\end{equation}
The inverse boundary value problem (IBVP) concerns recovery of the wave speed $c(x)$ and/or the potential $q(x)$ from knowledge of the ND map $\Lambda_{c,q}$, that is to invert the nonlinear map $(c,q)\mapsto \Lambda_{c,q}$.

\textbf{Literature.} The IBVP has been studied in the mathematical literature for a long time. 
The first coefficient determination result for a multidimensional wave equation, given the ND map, is by Rakesh and Symes~\cite{rakesh1988uniqueness}. They proved that a spatially varying potential is uniquely determined by the ND map in the case that $c= 1$. To our knowledge their method has never been implemented computationally.   

For variable $c$ and $q= 0$, Belishev~\cite{belishev1988approach} proved that the leading order coefficient $c$ is uniquely determined using the boundary control (BC) method combined with Tataru's unique continuation result~\cite{tataru1995unique}. The method has since been extended to 
many wave equations.
We mention~\cite{belishev1992reconstruction} for a generalization to Riemannian manifolds, and~\cite{MR1331288} for a result covering all symmetric time-independent lower order perturbations of the wave operator. 
Non-symmetric, time-dependent and matrix-valued lower order perturbations were recovered in~\cite{eskin2007inverse} \cite{MR1784415}, and~\cite{MR3880231}, respectively. For a review of the method, we refer to~\cite{MR2353313, MR1889089}.

The BC method has been implemented numerically to reconstruct the wave speed in~\cite{belishev1999dynamical}, and subsequently in~\cite{belishev2016numerical, de2018recovery, pestov2010numerical, yang2021stable}. 
The implementations~\cite{belishev1999dynamical, belishev2016numerical, de2018recovery} involve solving unstable control problems, whereas~\cite{pestov2010numerical, yang2021stable} are based on solving stable control problems but with target functions exhibiting exponential growth or decay. The exponential behaviour leads to instability as well. On the other hand, the linearized approach introduced in the present paper is stable. 
It should be mentioned that the BC method can be implemented in a stable way in the one-dimensional case, see~\cite{korpela2018discrete}.
For an interesting application of a variant of the method in the one-dimensional case, see~\cite{blaasten2019blockage} on detection of blockage in networks.

Under suitable geometric assumptions, it can be proven that the problem to recover the speed of sound is H\"older stable \cite{stefanov1998stability, stefanov2005stable}, even when the speed is given by an anisotropic Riemannian metric. Moreover, a low-pass version of $c$ can be recovered in a Lipschitz stable manner~\cite{liu2016lipschitz}.
The problem to recover $q$ is H\"older stable assuming, again, that the geometry is nice enough \cite{bellassoued2010stability, montalto2014stable, sun1990continuous}.
To our knowledge, the method, based on using high frequency solutions to the wave equation and yielding the latter three results, has not been implemented computationally. Stability results applicable to general geometries have been proven using the BC method in~\cite{MR2096795}, with an abstract modulus continuity, and very recently in~\cite{bosi2019reconstruction, burago2021stability}, with a doubly logarithmic modulus of continuity.

Finally, let us mention that recently there has been a lot of activity related to recovery of time-dependent coefficients in wave equations, see for example
\cite{stefanov2018lorentzian}, \cite{alexakis2021lorentzian}, and the references therein for up-to-date results. The first result in the time-dependent context was obtained in \cite{stefanov1989uniqueness}.

\textbf{Linearization.} \textit{Throughout the paper, we will assume the wave speed $c=c_0(x)$ is known}. Here $c_0(x) \in C^\infty(\overline{\Omega})$ is a smooth wave speed that is strictly positive. Then the IBVP only concerns recovery of the potential $q$.

We will study the linearization of this problem by linearzing the map $q\mapsto \Lambda_{c_0,q}$. For this purpose, we write
$$
q(x) = q_0(x) + \epsilon\dot{q}(x), \quad\quad u(t,x) = u_0(t,x) + \epsilon \dot{u}(t,x)
$$
where $q_0$ is a known background potential and $u_0$ is the background solution.
Insert these into~\eqref{eq:bvp}. Equating the $O(1)$-terms gives
\begin{equation} \label{eq:unperturbedBVP}
\left\{
\begin{array}{rcll}
\Box_{c_0,q_0} u_0(t,x) & = & 0, &\quad \text{ in } (0,2T) \times \Omega \\
 \partial_\nu u_0  & = & f, &\quad \text{ on } (0,2T) \times \partial\Omega \\ 
 u_0(0,x) = \partial_t u_0(0,x) & = & 0, &\quad x \in \Omega.
\end{array}
\right.
\end{equation}
Equating the $O(\epsilon)$-terms gives
\begin{equation} \label{eq:linearizedBVP}
\left\{
\begin{array}{rcll}
\Box_{c_0,q_0} \dot{u}(t,x) & = & u_0(t,x) \dot{q}(x), &\quad \text{ in } (0,2T) \times \Omega \\
 \partial_\nu \dot{u}  & = & 0, &\quad \text{ on } (0,2T) \times \partial\Omega \\ 
 \dot{u}(0,x) = \partial_t \dot{u}(0,x) & = & 0 &\quad x \in \Omega.
\end{array}
\right.
\end{equation}
Write the ND map $\Lambda_{q}=\Lambda_{q_0} + \epsilon \dot{\Lambda}_{\dot{q}}$, where $\Lambda_{q_0}$ is the ND map for the unperturbed boundary value problem~\eqref{eq:unperturbedBVP}, and $\dot{\Lambda}_{\dot{q}}$ is defined as
\begin{equation} \label{eq:linearizedNDmap}
\dot{\Lambda}_{\dot{q}}: f \mapsto \dot{u}|_{(0,2T)\times\partial\Omega}.
\end{equation}
Note that the unperturbed problem~\eqref{eq:unperturbedBVP} can be explicitly solved to obtain $u_0$ and $\Lambda_{q_0}$, since $c_0$ and $q_0$ are known. As before, we will write $\dot{u} = \dot{u}^f := \dot{\Lambda}_{\dot{q}} f$ if it is necessary to specify the Neumann data $f$.
Then the linearized IBVP concerns recovery of the potential $\dot{q}$ from $\dot{\Lambda}_{\dot{q}}$.

\textbf{Main Results.} The main contribution of this paper consists of several results regarding reconstruction of $\dot{q}$ based on the BC method. For constant $c_0$ and $q_0= 0$, we derive a reconstruction formula for $\dot{q}$ in Theorem~\ref{thm:recon1}, and prove that the reconstruction is of Lipschitz-type stability in Theorem~\ref{thm:stab1}.
For constant $c_0$ and $q_0\neq 0$, we derive a reconstruction formula in Theorem~\ref{thm:recon2}, and prove that the reconstruction is of H\"{o}lder-type stability in Theorem~\ref{thm:stab2}.
The formula is implemented and validated using numerical experiments, showing reliable reconstructions.

\textbf{Paper Structure.} The paper is organized as follows. Section~\ref{sec:prelim} reviews a few fundamental concepts and results in the boundary control theory. Section~\ref{sec:idandcontrol} consists of an integral identity and a controllability result that are essential to the development of our method. Section~\ref{sec:stabandrecon} establishes stability estimates and reconstruction formulae for the linearized IBVP, which are the central results of the paper. Section~\ref{sec:experiment} is devoted to numerical implementation of the reconstruction formulae as well as numerical experiments to assess performance of the proposed reconstructions.



\section{Preliminaries} \label{sec:prelim}

Introduce some notations: 
Given a function $u(t,x)$, we write $u(t)=u(t,\cdot)$ for the spatial part as a function of $x$.
Introduce the time reversal operator $R: L^2([0,T]\times\partial\Omega) \rightarrow L^2([0,T]\times\partial\Omega)$,
\begin{equation} \label{eq:R}
Ru(t,\cdot):=u(T-t,\cdot) ,\quad\quad 0<t<T;
\end{equation}
and the low-pass filter $J: L^2([0,2T]\times\partial\Omega) \rightarrow L^2([0,T]\times\partial\Omega)$
\begin{equation} \label{eq:J}
Jf(t,\cdot):=\frac{1}{2}\int^{2T-t}_t f(\tau,\cdot) \,d\tau,\quad\quad 0<t<T.
\end{equation}
We write $P_T: L^2((0,2T)\times\partial\Omega) \rightarrow L^2((0,T)\times\partial\Omega)$
for the orthogonal projection via restriction. Its adjoint operator $P^\ast_T: L^2((0,T)\times\partial\Omega) \rightarrow L^2((0,2T)\times\partial\Omega)$
is the extension by zero from $(0,T)$ to $(0,2T)$. Let $\mathcal{T}_D$ and $\mathcal{T}_N$ be the Dirichlet and Neumann trace operators respectively, that is,
$$
\mathcal{T}_D u(t,\cdot) = u(t,\cdot)|_{\partial\Omega}, \quad\quad\quad \mathcal{T}_N u(t,\cdot) = \partial_\nu u(t,\cdot)|_{\partial\Omega}.
$$

Introduce the \textit{connecting operator}
\begin{equation} \label{eq:K}
K:= J \Lambda_{q} P^\ast_T - R \Lambda_{q,T} R J P^\ast_T
\end{equation}
where $\Lambda_{q,T}f := P_T(\Lambda_{q} f)$.
Then the following Blagove\u{s}\u{c}enski\u{ı}’s identity holds~\cite{bingham2008iterative,blagoveshchenskii1967inverse,de2018exact,oksanen2013solving}. 

\bigskip
\begin{prop} \label{thm:waveinner}
Let $u^f, u^h$ be the solutions of \eqref{eq:bvp} with Neumann traces $f, h \in L^2((0,T)\times\partial\Omega)$, respectively. Then
\begin{equation} \label{eq:ufuh}
(u^f(T), u^h(T))_{L^2(\Omega,c_0^{-2}dx)} = (f,Kh)_{L^2((0,T)\times\partial\Omega)} = (Kf,h)_{L^2((0,T)\times\partial\Omega)}.
\end{equation}
\end{prop}

\begin{proof}
We first prove this for $f, h \in C^\infty_c((0,T)\times\partial\Omega)$.
Define 
$$
I(t,s) := (u^f(t), u^h(s))_{L^2(\Omega,c_0^{-2}dx)}.
$$
We compute
\begin{align}
 & (\partial^2_t - \partial^2_s) I(t,s) \nonumber \\
= & ((\Delta+q c_0^{-2}) u^f(t), u^h(s))_{L^2(\Omega)} - (u^f(t),(\Delta+q c_0^{-2}) u^h(s))_{L^2(\Omega)} \nonumber \\
= & (f(t), \Lambda_{q} P^*_T h(s))_{L^2(\partial\Omega)} - (\Lambda_{q} P^*_T f(t), h(s))_{L^2(\partial\Omega)}, \label{eq:id1}
\end{align}
where the last equality follows from integration by parts.
On the other hand, $I(0,s)=\partial_t I(0,s) = 0$ since $u^f(0,x)=\partial_t u^f(0,x) = 0$. Solve the inhomogeneous $1$D wave equation \eqref{eq:id1} together with these initial conditions to obtain 
\begin{align*}
I(T,T) & = \frac{1}{2} \int^T_0 \int^{2T-t}_{t} \left[ (f(t), \Lambda_{q} P^*_T h(\sigma))_{L^2(\partial\Omega)} - (\Lambda_{q} P^*_T f(t), h(\sigma))_{L^2(\partial\Omega)} \right] \,d\sigma dt \vspace{1ex}\\
 & = \int^T_0  [ ( f(t) ,  \frac{1}{2}\int^{2T-t}_{t} \Lambda_{q} P^*_T h(\sigma) \,d\sigma)_{L^2(\partial\Omega)} - ( \Lambda_{q} P^*_T f(t),  \frac{1}{2}\int^{2T-t}_{t} h(\sigma) \,d\sigma)_{L^2(\partial\Omega)} ] \,dt \vspace{1ex} \\
 & = (f, J \Lambda_{q} P^*_T h)_{L^2((0,T)\times\partial\Omega)} - (P_T(\Lambda_{q} f),J P^*_T h)_{L^2((0,T)\times\partial\Omega)}.
\end{align*}
Using the relations $P_T(\Lambda_{q} f) = \Lambda_{q,T}f$ and $\Lambda^\ast_{q,T} = R \Lambda_{q,T} R$ in $L^2((0,T)\times\partial\Omega)$, we have
\begin{align*}
I(T,T) & = (f,J \Lambda_{q} P^\ast_T h)_{L^2((0,T)\times\partial\Omega)} - (\Lambda_{q,T} f, J P^\ast_T h)_{L^2((0,T)\times\partial\Omega)} \\
 & = (f,J \Lambda_{q} P^\ast_T h)_{L^2((0,T)\times\partial\Omega)} - (f, R \Lambda_{q,T} R J P^\ast_T h)_{L^2((0,T)\times\partial\Omega)} \\
 & = (f,Kh)_{L^2((0,T)\times\partial\Omega)}
\end{align*}

For general $f, h \in L^2((0,T)\times\partial\Omega)$, simply notice that $K$ is a continuous operator and that compactly supported smooth functions are dense in $L^2$. The proof is completed.
\end{proof}

%
%

\bigskip

\begin{coro}
Suppose $f, h\in C^\infty_c((0,T]\times\partial\Omega)$. Then
\begin{equation} \label{eq:ufuhpp}
(c_0^2 \Delta u^{f}(T) - q u^{f}(T), u^h(T))_{L^2(\Omega,c_0^{-2}dx)} = (\partial^2_t f,Kh)_{L^2((0,T)\times\partial\Omega)} = (K \partial^2_t f,h)_{L^2((0,T)\times\partial\Omega)}.
\end{equation}
\end{coro}
\begin{proof}
Replacing $f$ by $\partial^2_t f$ in~\eqref{eq:ufuh}, we get
$$
(u^{\partial^2_t f}(T), u^h(T))_{L^2(\Omega,c_0^{-2}dx)} = (\partial^2_t f,Kh)_{L^2((0,T)\times\partial\Omega)} = (K\partial^2_t f,h)_{L^2((0,T)\times\partial\Omega)}.
$$
As both $u^{\partial^2_t f}$ and $\partial^2_t u^f$ satisfy~\eqref{eq:bvp} with $f$ replaced by $\partial^2_t f$, they must be equal thanks to the well-posedness of the boundary value problem. We conclude
$$
u^{\partial^2_t f} = \partial^2_t u^f
= c_0^2 \Delta u^f - q u^f.
$$
\end{proof}


\bigskip\bigskip
Recall that we write $\Lambda_{q}=\Lambda_{q_0} + \epsilon \dot{\Lambda}_{\dot{q}}$ in the linearization setting. Accordingly, we write $K=K_0+\epsilon\dot{K}$. Here
$K_0$ is the connecting operator for the background medium:
\begin{equation} \label{eq:K0}
K_0:= J \Lambda_{q_0} P^\ast_T - R \Lambda_{q_0,T} R J P^\ast_T.
\end{equation}
$K_0$ can be explicitly computed since $\Lambda_{c,q_0}$ is known.
$\dot{K}$ is the resulting perturbation in the connecting operator:
\begin{equation} \label{eq:Kdot}
\dot{K}:= J \dot{\Lambda}_{\dot{q}} P^\ast_T - R \dot{\Lambda}_{\dot{q},T} R J P^\ast_T.
\end{equation}
$\dot{K}$ can be explicitly computed once $\dot{\Lambda}_{\dot{q}}$ is given.

\section{Integral Identity and Controllability} \label{sec:idandcontrol}

First, we derive an integral identity that is essential to the development of the reconstruction procedure. We write $\dot{\Lambda}$ for $\dot{\Lambda}_{\dot{q}}$ when there is no risk of confusion.

\begin{prop} \label{thm:id}
Let $\lambda\in\mathbb{R}$ be a real number. If $f, h\in C^\infty_c((0,T]\times\partial\Omega)$ satisfy
\begin{equation} \label{eq:fhvanish}
[\Delta - q_0 c_0^{-2} +\lambda c_0^{-2}] u_0^{f}(T) =
[\Delta - q_0 c_0^{-2} + \lambda c_0^{-2}] u_0^h(T) = 0
\quad \text{ in } \Omega,
\end{equation}
then the following identity holds:
\begin{equation} \label{eq:keyid}
-(\dot{q} c_0^{-2} u_0^{f}(T), u^h_0(T))_{L^2(\Omega)} = (\partial^2_t f + \lambda f,\dot{K}h)_{L^2((0,T)\times\partial\Omega)} + (\dot{\Lambda} f(T), h(T))_{L^2(\partial\Omega)}
\end{equation}
\end{prop}

\begin{proof}

For $f, h\in C^\infty_c((0,T)\times\partial\Omega)$, we will make use of~\eqref{eq:ufuh} ~\eqref{eq:ufuhpp} to obtain some identities. First, let $\epsilon=0$ in~\eqref{eq:ufuh} we obtain
$$
(u_0^f(T), u_0^h(T))_{L^2(\Omega,c_0^{-2}dx)} = (f,K_0h)_{L^2((0,T)\times\partial\Omega)} = (K_0f,h)_{L^2((0,T)\times\partial\Omega)}.
$$
Next, differentiating~\eqref{eq:ufuh} in $\epsilon$ and let $\epsilon=0$, we obtain
\begin{align}
 & (f,\dot{K}h)_{L^2((0,T)\times\partial\Omega)} = (\dot{K}f,h)_{L^2((0,T)\times\partial\Omega)} \nonumber \\
= & (\dot{u}^f(T), u_0^h(T))_{L^2(\Omega,c_0^{-2}dx)} + (u_0^f(T), \dot{u}^h(T))_{L^2(\Omega,c_0^{-2}dx)}  \label{eq:ufuhlinearized}.
\end{align}

Similarly, let $\epsilon=0$ in~\eqref{eq:ufuhpp} we obtain
$$
(c_0^2 \Delta u_0^{f}(T) - q u_0^{f}(T), u_0^h(T))_{L^2(\Omega,c_0^{-2}dx)} = (\partial^2_t f,K_0h)_{L^2((0,T)\times\partial\Omega)} = (K_0 \partial^2_t f,h)_{L^2((0,T)\times\partial\Omega)}.
$$
Next, differentiating~\eqref{eq:ufuhpp} in $\epsilon$ and let $\epsilon=0$, we obtain
\begin{align}
 & (\partial^2_t f,\dot{K}h)_{L^2((0,T)\times\partial\Omega)} = (\dot{K} \partial^2_t f,h)_{L^2((0,T)\times\partial\Omega)} \nonumber \\
 = & (c_0^2 \Delta \dot{u}^{f}(T) - \dot{q} u_0^{f}(T) - q_0 \dot{u}^{f}(T), u_0^h(T))_{L^2(\Omega,c_0^{-2}dx)} \nonumber \\
+ & (c_0^2 \Delta u_0^{f}(T) - q_0 u_0^{f}(T), \dot{u}^h(T))_{L^2(\Omega,c_0^{-2}dx)} \nonumber \\
= & \underbrace{(\Delta \dot{u}^{f}(T), u_0^h(T))_{L^2(\Omega)}
- (q_0 c_0^{-2} \dot{u}^{f}(T), u_0^h(T))_{L^2(\Omega)} }_{:= I_1} \nonumber \\
 - & (\dot{q} c_0^{-2} u_0^{f}(T), u_0^h(T))_{L^2(\Omega)} \nonumber \\
+ & \underbrace{ (\Delta u_0^{f}(T), \dot{u}^h(T))_{L^2(\Omega)}
- (q_0 c_0^{-2} u_0^{f}(T), \dot{u}^h(T))_{L^2(\Omega)} }_{:= I_2} \label{eq:fiveterms}
\end{align}
where $L^2(\Omega) = L^2(\Omega,dx)$ is the $L^2$-space equipped with the usual Lebesgue measure.

For $I_1$, we integrate the first term by parts and use the fact that $\partial_\nu \dot{u}=0$ and $\dot{u}^{f}|_{(0,2T)\times\partial\Omega} = \dot{\Lambda} f$ to get
\begin{align*}
I_1 = & (\Delta \dot{u}^{f}(T), u_0^h(T))_{L^2(\Omega)}
-  (q_0 c_0^{-2} \dot{u}^{f}(T), u_0^h(T))_{L^2(\Omega)} \\
= & (\dot{u}^{f}(T), \Delta u_0^h(T))_{L^2(\Omega)} - 
(\dot{\Lambda} f(T), \partial_\nu u_0^h(T))_{L^2(\partial\Omega)} - 
(q_0 c_0^{-2} \dot{u}^{f}(T), u_0^h(T))_{L^2(\Omega)} \\
= & (\dot{u}^{f}(T), [\Delta - q_0 c_0^{-2}] u_0^h(T))_{L^2(\Omega)} - 
(\dot{\Lambda} f(T), h(T))_{L^2(\partial\Omega)}.
\end{align*}
On the other hand, combing the terms in $I_2$ we have
$$
    I_2 = ( [\Delta - q_0 c_0^{-2}] u_0^{f}(T), \dot{u}^h(T))_{L^2(\Omega)}.
$$
Insert these expressions for $I_1$ and $I_2$ into~\eqref{eq:fiveterms}, then add~\eqref{eq:ufuhlinearized} multiplied by $\lambda\in\mathbb{R}$ to get
\begin{align}
 & (\partial^2_t f + \lambda f, \dot{K}h)_{L^2((0,T)\times\partial\Omega)} + (\dot{\Lambda} f(T), h(T))_{L^2(\partial\Omega)} \nonumber \\
= & (\dot{u}^{f}(T), [\Delta - q_0 c_0^{-2} + \lambda c_0^{-2}] u_0^h(T))_{L^2(\Omega)} 
- (\dot{q} c_0^{-2} u_0^{f}(T), u_0^h(T))_{L^2(\Omega)} 
+ ( [\Delta - q_0 c_0^{-2} +\lambda c_0^{-2}] u_0^{f}(T), \dot{u}^h(T))_{L^2(\Omega)}.
\end{align}
If $[\Delta - q_0 c_0^{-2} +\lambda c_0^{-2}] u_0^{f}(T) =
[\Delta - q_0 c_0^{-2} + \lambda c_0^{-2}] u_0^h(T) = 0
\text{ in } \Omega$, the first term and last term on the right-hand side vanish, resulting in~\eqref{eq:keyid}.
\end{proof}

\bigskip
Next, we establish a boundary control estimate. Given a strictly positive $c_0\in C^\infty(\overline{\Omega})$, we will write $g := c_0^{-2} dx^2$ for the Riemannian metric associated to $c_0$, and denote by $S\overline{\Omega}$ the unit sphere bundle over the closure $\overline{\Omega}$ of $\Omega$.

\begin{prop} \label{thm:control}
Let $c_0\in C^\infty(\overline{\Omega})$ be strictly positive and $q_0\in C^\infty(\overline{\Omega})$.
Suppose that all maximal\footnote{For a maximal geodesic $\gamma : [a, b] \to \overline\Omega$ there may exists $t \in (a,b)$ such that $\gamma(t) \in \p \Omega$. The geodesics are maximal on the closed set $\overline{\Omega}$.} geodesics on $(\overline{\Omega}, g)$ have length strictly less than some fixed $T > 0$.  
Then for any $\phi \in C^\infty(\overline{\Omega})$ there is $f \in C_c^\infty((0,T] \times \p \Omega)$ such that 
    \begin{equation}\label{eq:controleqn}
u_0^f(T) = \phi \quad \text{in $\Omega$},
    \end{equation}
where $u_0$ is the solution of~\eqref{eq:unperturbedBVP}.
Moreover, there is $C>0$, independent of $\phi$, such that
    \begin{equation}\label{control_estimate}
\norm{f}_{H^2((0,T) \times \p \Omega)} 
\le C \norm{\phi}_{H^4(\Omega)}.
    \end{equation}
\end{prop} 
\begin{proof}
There is small $\delta > 0$ such that 
the maximal geodesics have length less than $T^* := T - 2 \delta$.
We extend $c_0$ and $q_0$ smoothly to $\R^n$.
Then there is a compact domain $\mathcal{K}$ with smooth boundary 
such that $\overline{\Omega}$ is contained in the interior of $\mathcal{K}$ and that the extended tensor $g = c_0^{-2} dx$ gives a Riemannian metric on $\mathcal{K}$.
Let $x \in \overline{\Omega}$ and let $\xi \in S_x \overline{\Omega}$, that is, $\xi$ is a unit vector with respect to $g$ at $x$. Write $\gamma_{x,\xi} : [a, b] \to \overline{\Omega}$ for the maximal geodesic with the initial data $(x,\xi)$. Then $b < T^*$.
We extend $\gamma_{x,\xi}$ as the maximal geodesic in $\mathcal{K}$
and write $\gamma_{x,\xi} : [\hat a, \hat b] \to \mathcal{K}$.
Then there are $t_j > b$ such that $t_j \to b$ and $\gamma_{x,\xi}(t_j) \notin \overline{\Omega}$. 
If $\hat b < T^*$ we extend $\gamma_{x,\xi}$ to $[\hat a, T^*]$ by setting 
$\gamma_{x,\xi}(t) = \gamma_{x,\xi}(\hat b)$ for $t > \hat b$.
Then the function 
    \begin{align*}
\rho(x,\xi) := \max \{d(\gamma_{x,\xi}(t), \overline{\Omega}) : t \in [0, T^*]\}
    \end{align*}
satisfies $\rho(x,\xi) > 0$.
Let us show that $\rho$ is continuous.
It follows from the smoothness of the geodesic flow and the triangle inequality that the function 
    \begin{align*}
S\overline{\Omega} \times [0,T^*] \times \overline{\Omega} \ni (x,\xi,t,y) \mapsto d(\gamma_{x,\xi}(t), y)
    \end{align*}
is uniformly continuous. By Lemma~\ref{thm:unicont} in the appendix, the function
    \begin{align*}
S\overline{\Omega} \times [0,T^*] \ni (x,\xi,t) \mapsto d(\gamma_{x,\xi}(t), \overline{\Omega}),
    \end{align*}
is uniformly continuous, and so is $\rho$. 

We have show that the continuous function $\rho$ is strictly positive on the compact set $S \overline{\Omega}$. Thus there is an open set $\omega_0$
such that $\overline{\omega_0} \subset \mathcal{K} \setminus \overline{\Omega}$
and that all geodesics $\gamma_{x,\xi}$ with $(x,\xi) \in S\overline{\Omega}$
intersect $\omega_0$ in time $T^*$. 

We choose $\eta \in C_c^\infty(0,T)$ taking values in $[0,1]$ and satisfying $\eta = 1$ on $[\delta, T-\delta]$,
and $\chi \in C^\infty_c(\mathcal{K} \setminus \overline{\Omega})$
satisfying $\chi = 1$ on $\omega_0$.
Finally, we choose an extension $\phi \in C_c^\infty(\mathcal{K})$.
Now \cite[Theorem 5.1]{ervedoza2010systematic} implies that
there is $Y \in C^\infty((0,T) \times (\mathcal{K} \setminus \overline{\Omega}))$
such that 
    \begin{align}\label{eq_control}
\begin{cases}
\Box_{c_0,q_0} v = \eta \chi Y,
\\
v|_{x \in \p \mathcal{K}} = 0,
\\
v|_{t=T} = \phi,\ \p_t v|_{t=T} = 0,
\\
v|_{t=0} = 0,\ \p_t v|_{t=0} = 0.
\end{cases} 
    \end{align}
We extend $v$ to $(0,2T) \times \mathcal{K}$ by solving
    \begin{align*}
\begin{cases}
\Box_{c_0,q_0} v = 0,
\\
v|_{x \in \p \mathcal{K}} = 0,
\\
v|_{t=T} = \phi,\ \p_t v|_{t=T} = 0,
\end{cases} 
    \end{align*}
and set $f = \p_\nu v|_{x \in \p \Omega}$.
As $\eta(t)  = 0$ when $t$ is close to zero 
and as $v$ satisfies vanishing initial conditions at $t=0$,
also $f(t,x) = 0$ for $t$ near zero. 

It remains to show estimate (\ref{control_estimate}).
The extension of $\phi$ can be chosen so that 
    \begin{align*}
\norm{\phi}_{H^4(\mathcal{K})} \le C \norm{\phi}_{H^4(\Omega)}.
    \end{align*}
By \cite[Theorem 5.1]{ervedoza2010systematic}
there is a map  
    \begin{align}\label{Y_map}
G : H^4(\mathcal{K}) \to H^3(\mathcal{K}) \times H^2(\mathcal{K}),
\quad
G(\phi) = (Y|_{t=0},\p_t Y|_{t=0}),
    \end{align}
and $Y$ satisfies
    \begin{align*}
\begin{cases}
\Box_{c_0,q_0} Y = 0, 
\\
Y|_{x \in \p \mathcal{K}}=0.
\end{cases}
    \end{align*}
As continuity of (\ref{Y_map}) is not explicitly stated in \cite{ervedoza2010systematic}, we will show the continuity for the convenience of the reader.

Due to the closed graph theorem it is enough to show that
if $\phi_j \to \phi$ in $H^4(\mathcal{K})$
and $G(\phi_j) \to (Y_0, Y_1)$ in $H^3(\mathcal{K}) \times H^2(\mathcal{K})$
then $G(\phi) = (Y_0, Y_1)$. We define $v_j$ and $Y_j$ by
    \begin{align*}
\begin{cases}
\Box_{c_0,q_0} v_j = \eta\chi Y_j, 
\\
v_j|_{x \in \p \mathcal{K}}=0,
\\
v_j|_{t=0} = 0,\ \p_t v_j|_{t=0} = 0,
\end{cases}
\qquad
\begin{cases}
\Box_{c_0,q_0} Y_j = 0, 
\\
Y_j|_{x \in \p \mathcal{K}}=0,
\\
(Y_j|_{t=0}, \p_t Y_j|_{t=0}) = F(\phi_j),
\end{cases}
    \end{align*}
Then $v_j \to v$ and $Y_j \to Y$ where $v$ and $Y$ solve
    \begin{align*}
\begin{cases}
\Box_{c_0,q_0} v = \eta\chi Y, 
\\
v|_{x \in \p \mathcal{K}}=0,
\\
v|_{t=0} = 0,\ \p_t v|_{t=0} = 0,
\end{cases}
\qquad
\begin{cases}
\Box_{c_0,q_0} Y = 0, 
\\
Y|_{x \in \p \mathcal{K}}=0,
\\
(Y|_{t=0}, \p_t Y|_{t=0}) = (Y_0, Y_1),
\end{cases}
    \end{align*}
and also 
    \begin{align*}
v_j|_{t=T} = \lim_j v_j|_{t=T} = \lim_j \phi_j = \phi,
\quad \p_t v|_{t=T} = \lim_j \p_t v_j|_{t=T} = 0.
    \end{align*}
But the control given by \cite[Theorem 5.1]{ervedoza2010systematic}
for $\phi$ is characterized by the following system 
    \begin{align*}
\begin{cases}
\Box_{c_0,q_0} v = \eta\chi Y, 
\\
v|_{x \in \p \mathcal{K}}=0,
\\
v|_{t=0} = \phi,\ \p_t u|_{t=0} = 0,
\\
v|_{t=T} = 0,\ \p_t u|_{t=T} = 0,
\end{cases}
\qquad
\begin{cases}
\Box_{c_0,q_0} Y = 0, 
\\
Y|_{x \in \p \mathcal{K}}=0,
\end{cases}
    \end{align*}
see~\cite{burman2021spacetime}. Hence the control $Y$ that we obtained as the limit coincides with that given by \cite[Theorem 5.1]{ervedoza2010systematic},
and $G(\phi) = (Y_0, Y_1)$.

By \cite[Theorem 5.1]{ervedoza2010systematic}
the initial data $G(\phi)$ satisfies suitable compatibility conditions so that $Y \in H^3((0,T) \times \mathcal{K})$.
Moreover, solving (\ref{eq_control}) from the initial data at $t=0$, while discarding the constaint at $t=T$, gives
    \begin{align*}
\norm{v}_{H^4((0,T) \times \mathcal{K})} 
\le C \norm{Y}_{H^3((0,T) \times \mathcal{K})} 
\le C \norm{G(\phi)}_{H^3(\mathcal{K}) \times H^2(\mathcal{K})}
\le C \norm{\phi}_{H^4(\mathcal{K})}.
    \end{align*}
\end{proof}

\textbf{Remark:}
Solvability of an equation like~\eqref{eq:controleqn} is a central question in the boundary control theory. There are other results which ensure the solvability in other function spaces. For example, if we define
$$
W_0: L^2((0,T)\times\partial\Omega)\rightarrow L^2(\Omega, c_0^{-2}dx), \quad\quad\quad W_0 f=u^f(T).
$$
This clearly is a bounded linear operator. Moreover, if $T>0$ is large enough, the range $Ran(W_0)$ is dense in $L^2(\Omega)$ by Tataru’s unique continuation~\cite{tataru1995unique}.
If we further assume that the continuous observability condition~\cite{liu2016lipschitz} holds for the background wave solution, then $W_0$ is surjective, that is, $Ran(W_0) = L^2(\Omega)$.
This ensures the existence of solutions $f,h \in L^2((0,T)\times\partial\Omega)$ for the equation~\eqref{eq:normal}.
However, such $L^2$-regularity is insufficient for our purpose, as we need $\partial^2_t f$ to exist, see~\eqref{eq:reconformula}.

\section{Stability and Reconstruction} \label{sec:stabandrecon}

We will make use of Proposition~\ref{thm:id} and Proposition~\ref{thm:control} to derive stability estimate and reconstruction formulae for $\dot{q}$, on the premise that the wave speed is constant. Without loss of generality, we take $c_0=1$. The discussion is separate for vanishing and non-vanishing background potentials $q_0$.

\subsection{Case 1: $c_0= 1$ and $q_0= 0$} \label{sec:q0vanish}
We take $\lambda\geq 0$ and dimension $n\geq 1$, then the equation~\eqref{eq:fhvanish} for $u^f_0(T)$ and $u^h_0(T)$ becomes the Helmholtz equation
$$
(\Delta +\lambda) u_0^{f}(T) =
(\Delta + \lambda) u_0^h(T) = 0
\quad \text{ in } \Omega.
$$

A class of Helmholtz solutions are the plane waves $e^{i \sqrt{\lambda}\theta\cdot x}$ where $\theta \in\mathbb{S}^{n-1}$ is an arbitrary unit vector.
Moreover, Proposition~\ref{thm:control} ensures the existence of $f,h\in C^\infty_c((0,T]\times\partial\Omega)$ such that
\begin{equation} \label{eq:normal}
u^f_0(T) = u^h_0(T) = e^{i \sqrt{\lambda}\theta\cdot x}.
\end{equation}

\bigskip
\begin{thm} \label{thm:recon1}
Suppose $c_0= 1$ and $q_0= 0$. Then the Fourier transform $\hat{\dot{q}}$ of $\dot{q}$ can be reconstructed as follows:
\begin{equation} \label{eq:reconformula}
\hat{\dot{q}}(2\sqrt{\lambda}\theta) = - (\partial^2_t f + \lambda f,\dot{K}h)_{L^2((0,T)\times\partial\Omega)} - (\dot{\Lambda} f(T), h(T))_{L^2(\partial\Omega)}
\end{equation}
where $f,h \in C^\infty_c((0,T]\times\partial\Omega)$ are solutions to~\eqref{eq:normal}.
\end{thm}
\begin{proof}
The formula is obtained by inserting $u^f_0(T) = u^h_0(T) = e^{i \sqrt{\lambda}\theta\cdot x}$ into~\eqref{eq:keyid}.
As $\lambda \geq 0$ and $\theta$ are arbitrary, it recovers the Fourier transform of $\dot{q}$ everywhere.
\end{proof}

\begin{remark}
An explicit procedure to solve for $f$ and $h$ from~\eqref{eq:normal} is explained in Section~\ref{sec:BC}.
\end{remark}

\bigskip
Next, we show that the reconstruction above has Lipschitz-type stability. As the inverse problem is linear, it suffices to bound $\dot{q}$ using continuous functions of $\dot{\Lambda}$.

\begin{thm} \label{thm:stab1}
Suppose $c_0= 1$ and $q_0= 0$. 
There exists a constant $C>0$, independent of $\lambda$, such that 
$$
\left| \hat{\dot{q}}(\sqrt{2\lambda}\theta) \right| \leq C (2+\lambda)\lambda^4  
\left(
   \|\dot{K}\|_{L^2((0,T)\times\partial\Omega) \rightarrow L^2((0,T)\times\partial\Omega)} + 
   \|\dot{\Lambda}\|_{H^2((0,T)\times\partial\Omega) \rightarrow H^3((0,T)\times\partial\Omega)}
   \right)
$$
Here $\dot{K}$ is viewed as a linear function of $\dot{\Lambda}$ as is defined in~\eqref{eq:Kdot}. \label{thm:nopotential}
\end{thm}

\begin{proof}
For a bounded linear operator $T: \mathcal{X}\rightarrow\mathcal{Y}$ between two Hilbert spaces $\mathcal{X}$ and $\mathcal{Y}$, we write $\|T\|_{\mathcal{X}\rightarrow\mathcal{Y}}$ for the operator norm of $T$. 
Let $f,h\in C^\infty_c((0,T]\times\partial\Omega)$ be solutions of~\eqref{eq:normal} obtained from Proposition~\ref{thm:control}. We employ~\eqref{eq:reconformula} to estimate: 
\begin{align*}
\left| \mathcal{F}[\dot{q} ](\sqrt{2\lambda}\theta) \right| \leq & 
\|\partial^2_t f + \lambda f \|_{L^2((0,T)\times\partial\Omega)} 
\|h \|_{L^2((0,T)\times\partial\Omega)} 
    \|\dot{K}\|_{L^2((0,T)\times\partial\Omega) \rightarrow L^2((0,T)\times\partial\Omega)} \\
+ & \|\dot{\Lambda}f(T)\|_{L^2(\partial\Omega)} \|h(T)\|_{L^2(\partial\Omega)} \\
\leq & (1+\lambda) \|f \|_{H^2((0,T)\times\partial\Omega)}
\|h \|_{L^2((0,T)\times\partial\Omega)} 
    \|\dot{K}\|_{L^2((0,T)\times\partial\Omega) \rightarrow L^2((0,T)\times\partial\Omega)} \\
+ & \|\dot{\Lambda}f\|_{H^1((0,T)\times\partial\Omega)} \|h\|_{H^1((0,T)\times\partial\Omega)}
\end{align*}
by the continuity of the trace operator.

It remains to estimate $\|\dot{\Lambda}f\|_{H^1((0,T)\times\partial\Omega)}$. For this purpose, we extend $f\in H^{2}((0,T)\times\partial\Omega)$ to a function $F\in H^{2+\frac{3}{2}}((0,T)\times\Omega)$ so that $\partial_\nu F|_{(0,T)\times\partial\Omega}=f$ and $F(t,x)=0$ for $x\in\Omega$ and $t$ close to $0$ (recall that $f(t,x)=0$ for $t$ near $0$). Such $F$ can be chosen to fulfill
$$
    \|F\|_{H^{3+\frac{1}{2}}((0,T)\times\Omega)} 
    \leq C \|f\|_{H^{2}((0,T)\times\partial\Omega)}
$$
Set $v:=F-u_0$ where $u_0$ is the solution of~\eqref{eq:unperturbedBVP}, then $v$ satisfies
$$
\left\{
\begin{array}{rcll}
\Box_{c_0,q_0} v & = & \Box_{c_0,q_0} F, &\quad \text{ in } (0,2T) \times \Omega \\
 \partial_\nu v  & = & 0, &\quad \text{ on } (0,2T) \times \partial\Omega \\ 
 v|_{t=0} = \partial_t v|_{t=0} & = & 0, &\quad x \in \Omega.
\end{array}
\right.
$$
The regularity estimate for the wave equation~\cite{evans1998partial} implies
$$
    \|v\|_{H^{2+\frac{1}{2}}((0,T)\times\Omega)} 
    \leq C \|\Box_{c_0,q_0} F\|_{H^{1+\frac{1}{2}}((0,T)\times\Omega)}  
    \leq C \|F\|_{H^{3+\frac{1}{2}}((0,T)\times\Omega)}
$$
We conclude $u_0\in H^{2+\frac{1}{2}}((0,T)\times\Omega)$ and $\dot{q} u_0 \in H^{2+\frac{1}{2}}((0,T)\times\Omega)$.
The same regularity estimate for the wave equation applied to~\eqref{eq:linearizedBVP} implies
$$
    \|\dot{u}\|_{H^{3+\frac{1}{2}}((0,T)\times\Omega)} 
    \leq C \|\dot{q} u_0\|_{H^{2+\frac{1}{2}}((0,T)\times\Omega)}  
    \leq C \|u_0\|_{H^{2+\frac{1}{2}}((0,T)\times\Omega)}  
$$
These inequalities together with the trace estimate yield
$$
    \|\dot{\Lambda}f\|_{H^{3}((0,T)\times\partial\Omega)} 
    \leq C \|\dot{u}\|_{H^{3+\frac{1}{2}}((0,T)\times\Omega)}
    \leq C \|f\|_{H^{2}((0,T)\times\partial\Omega)} 
$$
where the constant $C>0$ is independent of $f$. Hence $\dot{\Lambda}:H^{2}((0,T)\times\partial\Omega) \rightarrow H^{3}((0,T)\times\partial\Omega)$ is a bounded linear operator.

Finally, we can complete the stability estimate:
\begin{align*}
  \left| \mathcal{F}[ \dot{q} ](\sqrt{2\lambda}\theta) \right| \leq & (1+\lambda) \|f \|_{H^2((0,T)\times\partial\Omega)}
    \|h \|_{H^1((0,T)\times\partial\Omega)} 
    \|\dot{K}\|_{L^2((0,T)\times\partial\Omega) \rightarrow L^2((0,T)\times\partial\Omega)} \\
   + & \|\dot{\Lambda}f\|_{H^3((0,T)\times\partial\Omega)} \|h\|_{H^1((0,T)\times\partial\Omega)}  \\
   \leq & (1+\lambda) \|f \|_{H^2((0,T)\times\partial\Omega)}
    \|h \|_{H^1((0,T)\times\partial\Omega)} 
    \|\dot{K}\|_{L^2((0,T)\times\partial\Omega) \rightarrow L^2((0,T)\times\partial\Omega)} \\
   + & \|f\|_{H^2((0,T)\times\partial\Omega)}
   \|\dot{\Lambda}\|_{H^2((0,T)\times\partial\Omega) \rightarrow H^3((0,T)\times\partial\Omega)}
   \|h\|_{H^1((0,T)\times\partial\Omega)} \\
   \leq & C_\lambda \left(
   \|\dot{K}\|_{L^2((0,T)\times\partial\Omega) \rightarrow L^2((0,T)\times\partial\Omega)} + 
   \|\dot{\Lambda}\|_{H^2((0,T)\times\partial\Omega) \rightarrow H^3((0,T)\times\partial\Omega)}
   \right)
\end{align*}
where the constant $C_\lambda$ satisfies (see~\eqref{control_estimate})
\begin{align*}
    C_\lambda & := (2+\lambda) \|f \|_{H^2((0,T)\times\partial\Omega)}
    \|h \|_{H^1((0,T)\times\partial\Omega)} \\
    \leq & C (2+\lambda) \|e^{i \sqrt{\lambda}\theta\cdot x} \|^2_{H^4(\Omega)} \leq C (2+\lambda)\lambda^4.
\end{align*}
for some constant $C>0$ independent of $\lambda$.
\end{proof}

\subsection{Case 2: $c_0= 1$ and $q_0\neq 0$}

Let $\lambda \geq 0$ again, then the equations for $u^f_0(T)$ and $u^h_0(T)$ become the perturbed Helmholtz equation
$$
[\Delta + \lambda - q_0] u_0^{f}(T) =
[\Delta + \lambda - q_0] u_0^h(T) = 0
\quad \text{ in } \Omega.
$$
A class of solutions are total waves of the form
\begin{equation} \label{eq:phi}
\phi(x) = e^{i\sqrt{\lambda} \theta\cdot x} + r(x;\lambda)
\end{equation}
with the scattered wave $r(x;\lambda)$ satisfying
\begin{equation} \label{eq:rdecay}
\|r\|_{H^s(\mathbb{R}^n)} = O(\lambda^\frac{s-1}{2}) \quad\quad \text{ as } \lambda\rightarrow\infty
\end{equation}
for any $s\geq 0$. Indeed, $r$ is the unique outgoing solution to~\eqref{eq:scatter}, see Lemma~\ref{thm:resolest} in the appendix for the construction and property of $r$.

Consider dimension $n\geq 2$ and let $\theta,\omega\in\mathbb{R}^n$ be two vectors such that $\theta\perp\omega$. We take the following solutions:
\begin{align*}
\phi(x) := & \phi_0(x) + r_1(x;\lambda), \quad\quad \phi_0(x) := e^{i (k\theta + l \omega)\cdot x} \\
\psi(x) := & \psi_0(x) + r_2(x;\lambda), \quad\quad \psi_0(x) := e^{i (k\theta - l \omega)\cdot x}
\end{align*}
where $r_1, r_2$ satisfy~\eqref{eq:rdecay}.
Choose $k^2+l^2=\lambda$ so that $(\Delta+\lambda) \phi_0 = (\Delta+\lambda) \psi_0 = 0$.
Proposition~\ref{thm:control} asserts that there are $f,h \in C^\infty_c((0,T]\times\partial\Omega)$ such that
\begin{equation} \label{eq:normal2}
u^f_0(T) =\phi = \phi_0 + r_1, \quad\quad  u^h_0(T) = \psi = \psi_0 + r_2.
\end{equation}

\bigskip
\begin{thm} \label{thm:recon2}
Suppose $c_0= 1$, $q_0\in C^\infty(\overline{\Omega})$ and $q_0$ is not identically zero. Then the Fourier transform $\hat{\dot{q}}$ can be reconstructed as follows:
$$
\hat{\dot{q}}(2k\theta) = - \lim_{l\rightarrow\infty} \left[ (\partial^2_t f + (k^2+l^2) f,\dot{K}h)_{L^2((0,T)\times\partial\Omega)} + (\dot{\Lambda} f(T), h(T))_{L^2(\partial\Omega)} \right].
$$
where $f,h \in C^\infty_c((0,T]\times\partial\Omega)$ are solutions to~\eqref{eq:normal2}, respectively.
\end{thm}
\begin{proof}
Inserting $u^f_0(T)=\phi$ and $u^h_0(T)=\psi$ into~\eqref{eq:keyid} gives
\begin{align}
 & - \hat{\dot{q}}(2k\theta) - (\dot{q} e^{i (k \theta+ l\omega)\cdot x}, r_2)_{L^2(\Omega)} 
- (\dot{q} e^{i (k \theta - l\omega)\cdot x}, r_1)_{L^2(\Omega)} 
- (\dot{q} r_1, r_2)_{L^2(\Omega)} \nonumber \\
=&  (\partial^2_t f + (k^2 + l^2),\dot{K}h)_{L^2((0,T)\times\partial\Omega)} + (\dot{\Lambda} f(T), h(T))_{L^2(\partial\Omega)} \label{eq:est1}
\end{align}
If we fix $k$ and let $l\rightarrow\infty$, then $\lambda\rightarrow\infty$ and $\|r\|_{L^2(\Omega)}\rightarrow 0$ due to its decay property~\eqref{eq:rdecay}. We obtain the reconstruction formula for any $k\geq 0$ and any $\theta\in\mathbb{S}^{n-1}$. Note that $f$ and $h$ depend on $l$.
\end{proof}

\bigskip

We can also obtain a H\"{o}lder-type stability estimate for $\|\dot{q}\|_{H^{-s}(\mathbb{R}^n)}$, where $s>0$ is an arbitrary real number and $H^{-s}(\mathbb{R}^n)$ is the $L^2$-based Sobolev space of order $-s$ over $\mathbb{R}^n$.

\begin{thm} \label{thm:stab2}
Suppose $c_0= 1$, $q_0\in C^\infty(\overline{\Omega})$ and $q_0$ is not identically zero. 
For any $s>0$, there exists a constant $C>0$ such that 
$$
\|\dot{q}\|_{H^{-s}(\mathbb{R}^n)} \leq C \left(
   \|\dot{K}\|_{L^2((0,T)\times\partial\Omega) \rightarrow L^2((0,T)\times\partial\Omega)} + 
   \|\dot{\Lambda}\|_{H^2((0,T)\times\partial\Omega) \rightarrow H^3((0,T)\times\partial\Omega)}
   \right)^{\frac{2s}{11(n+2s)}}.
$$
\end{thm}

\begin{proof}
Write $\xi:=2k\theta$ and 
$$
\delta := \left(
   \|\dot{K}\|_{L^2((0,T)\times\partial\Omega) \rightarrow L^2((0,T)\times\partial\Omega)} + 
   \|\dot{\Lambda}\|_{H^2((0,T)\times\partial\Omega) \rightarrow H^3((0,T)\times\partial\Omega)}
   \right).
$$
Let $\rho>0$ be a sufficiently large number that is to be determined. We decompose
$$
\|\dot{q}\|^2_{H^{-s}(\mathbb{R}^n)} = 
\int_{|\xi| \leq \rho} \frac{|\hat{\dot{q}}(\xi)|^2}{(1+ |\xi|^2)^{s}} \, d\xi
+ \int_{|\xi| > \rho} \frac{|\hat{\dot{q}}(\xi)|^2}{(1+ |\xi|^2)^{s}} \, d\xi.
$$
For the integral over high frequencies, we have
$$
\int_{|\xi| > \rho} \frac{|\hat{\dot{q}}(\xi)|^2}{(1+ |\xi|^2)^{s}} \, d\xi
\leq \frac{1}{(1+ \rho^2)^{s}} \int_{|\xi| > \rho} |\hat{\dot{q}}(\xi)|^2 \, d\xi
\leq \frac{\|\dot{q}\|^2_{L^2(\mathbb{R}^n)}}{(1+\rho^{2})^s} \leq C \frac{1}{\rho^{2s}}.
$$
For the integral over low frequencies, it is easy to see that:
$$
\int_{|\xi| \leq \rho} \frac{|\hat{\dot{q}}(\xi)|^2}{(1+ |\xi|^2)^s} \, d\xi 
\leq \int_{|\xi| \leq \rho} |\hat{\dot{q}}(\xi)|^2 \, d\xi 
\leq C \rho^n \|\hat{\dot{q}}\|^2_{L^\infty(B(0,\rho))}. 
$$
The norm $\|\hat{\dot{q}}\|_{L^\infty(B(0,\rho))}$ can be estimated using~\eqref{eq:est1}. Indeed, for $|\xi|\leq \rho$, we have
\begin{align*}
|\hat{\dot{q}}(\xi)| & \leq |(\partial^2_t f + (k^2 + l^2),\dot{K}h)_{L^2((0,T)\times\partial\Omega)} + (\dot{\Lambda} f(T), h(T))_{L^2(\partial\Omega)}| + \frac{C}{\sqrt{\lambda}} \\
& \leq C (2+\lambda) \|\phi\|_{H^4(\Omega)} \|\psi\|_{H^4(\Omega)} \delta + \frac{C}{\sqrt{\lambda}} \\
& \leq C (2+\lambda)
\left( \|\phi_0\|_{H^4(\Omega)} + \|r_1\|_{H^4(\Omega)} \right)
\left( \|\psi_0\|_{H^4(\Omega)} + \|r_2\|_{H^4(\Omega)} \right) \delta + \frac{C}{\sqrt{\lambda}}  \\
& \leq C (2+\lambda) \left(
\lambda^2 + \lambda^{\frac{3}{2}} \right)^2 \delta + \frac{C}{\sqrt{\lambda}}
\end{align*}
where the first inequality is a consequence of the $L^2$-resolvent estimate for $r_1,r_2$, the second inequality follows from the proof of Proposition~\ref{thm:nopotential}, and the last inequality utilizes the resolvent estimate for higher-order derivatives of $r_1,r_2$.
Utilizing the relation $\lambda=k^2+l^2$, we conclude
$$
\|\hat{\dot{q}}\|^2_{L^\infty(B(0,\rho))}
\leq C \left[ (2+\lambda)^2 \lambda^6 (1+\sqrt{\lambda})^4 \delta^2 + \frac{1}{\lambda} \right]
\leq C \left[ (\rho^2+l^2)^{10} \delta + \frac{1}{l^2} \right]
$$
provided $\rho>0$ is sufficiently large.
Combining these estimates, we see that
$$
\|\dot{q}\|^2_{H^{-s}(\mathbb{R}^n)}
\leq C \left[ \rho^n (\rho^2+l^2)^{10} \delta^2 + \frac{\rho^n}{l^2} + \frac{1}{\rho^{2s}} \right].
$$
Choosing $l = \rho^{n+2s}$ and $\rho = \delta^{-\frac{2}{21(n+2s)}}$ yields
$$
\|\dot{q}\|^2_{H^{-s}(\mathbb{R}^n)} \leq C \delta^{\frac{4s}{11(n+2s)}} 
$$
as long as $\delta$ is sufficiently small.

\end{proof}

\section{Numerical Experiments} \label{sec:experiment}

This section is devoted to numerical implementation and validation of the reconstruction formula~\eqref{eq:reconformula} in one dimension (1D) when $c_0= 1$ and $q_0= 0$.

\subsection{Computing Boundary Controls with Time Reversal} \label{sec:BC}

A crucial step of the boundary control method is solving the equations~\eqref{eq:normal} or \eqref{eq:normal2} for the boundary controls $f$ and $h$. For our purpose, a highly accurate numerical solver is needed due to the appearance of $\partial^2_t f$ in~\eqref{eq:reconformula}, where the second-order temporal differentiation tends to amplify the numerical error in $f$.
When the spatial dimension $n$ is odd with $c_0= 1$ and $q_0= 0$, such scheme can be obtained using time reversal. Indeed, for a prescribed $\phi$ on $\Omega$, one can construct an extension, named $\tilde{\phi}$, such that $\tilde{\phi}=\phi$ in $\Omega$ and $\tilde{\phi}$ is supported in a neighborhood of $\Omega$. Let $v$ be the solution of the following backward initial value problem
$$
\left\{
\begin{array}{rcl}
\Box_{1,0} v(t,x) & = 0, &\quad \text{ in } (0,T) \times \mathbb{R}^n \\
v(T) & = \tilde\phi, &\quad \text{ in } \mathbb{R}^n \\
\partial_t v(T) & = 0, &\quad \text{ in } \mathbb{R}^n.
\end{array}
\right.
$$

If $T>0$ is sufficiently large and $n$ is odd, we would have $v(0)=\partial_t v(0)=0$ by the Huygen's principle. This implies 
$$
 u_0^{f-\partial_\nu v}(T) = u^f_0(T) - v(T) = 0 \text{ in } \Omega.
$$
As a result, we can take $f=\partial_\nu v|_{[0,T]\times\partial\Omega}$. Note that $v$ can be explictly expressed using the Kirchhoff's formula~\cite{evans1998partial}, thus $\partial^2_t f= \partial_\nu \partial^2_t v|_{[0,T]\times\partial\Omega}$ can be analytically computed.

We will demonstrate the numerical implementations in dimension $n=1$ with $\Omega = [a,b]$. In this case, D'Alembert's formula gives
$$
v(t,x) = \frac{1}{2} [\tilde\phi(x+t-T) + \tilde\phi(x+T-t)].
$$
Therefore,
$$
\partial^2_t f=  \partial_\nu \partial^2_t v|_{[0,T]\times\partial\Omega} = \pm\frac{1}{2} [\tilde\phi'''(x+t-T) + \tilde\phi'''(x+T-t)]|_{x=a,b}
$$
where we take $+$ when $x=b$ and $-$ when $x=a$. We choose the following extension:
$$
    \tilde\phi :=
    \begin{cases}
    \phi&x\in[a,b],\\
    \phi\cdot\exp\{1-\frac{1}{1-(x-a)^{2p}}\}&x\in(a-1,a),\\
    \phi\cdot\exp\{1-\frac{1}{1-(x-b)^{2p}}\}&x\in(b,b+1),\\
    0&x\notin(a-1,b+1),
    \end{cases}
$$
where $p\geq 1$ is a positive integer. It is easy to see that $\tilde\phi$ is $C^{2p-1}$ at $x=a,b$ and is $C^\infty$ at other points. We take $p\geq2$ to guarantee the existence of the second derivative of $f$.

%

\subsection{Numerical Experiments}

The computational setup is as follows: $\Omega=[-1,1]$, $c_0=1$, $q_0=0$, and $T=5$. The forward problem~\eqref{eq:linearizedBVP} is solved using the second order central difference scheme on a temporal-spatial grid of size $24999\times501$ to obtain the linearized ND map~\eqref{eq:linearizedNDmap} . Then various $\phi=u^f_0(T)$ are inserted into~\eqref{eq:reconformula} to recover the Fourier transform of $\dot{q}$ at $2k\theta$, where $f$ and $h$ are computed using the time revesal method as is explained at the beginning of this section.
The basis functions for the prescribed Helmholtz solution $\phi$ in our experiments are
\[1,\sin\left(\frac{\pi}{2}x\right),\cos\left(\frac{\pi}{2}x\right),\dots,\sin\left(\frac{N\pi}{2}x\right),\cos\left(\frac{N\pi}{2}x\right)\]
with $N=10$. They correspond to Helmholtz solutions with $k=0, \frac{\pi}{2},\dots, \frac{N \pi}{2}$.

\bigskip
\textbf{Experiment 1.}
The first experiment aims to reconstruct the following smooth $\dot{q}$ using the formula~\eqref{eq:reconformula}:
\[\dot{q}=\sin(\pi x)+2\cos(2\pi x)+4\sin(4\pi x)-3.\]
The graph of $\dot{q}$ is shown in Figure~\ref{fig:c1}. The measurement $\dot{\Lambda}_{\dot{q}}$ is added with $0\%,1\%$, and $5\%$ of Gaussian noise, respectively. The reconstructions and the corresponding errors are illustrated in Figure~\ref{fig:exp1}. 
Notice that the reconstruction error with $5\%$ noise is relatively larger, as can be expected. When multiple measurements are available, we can repeat the reconstruction several times and then take the average. This strategy effectively reduces the error, since the inverse problem is linear and the Gaussian noise have zero mean, see Figure~\ref{fig:repeat1}.

\begin{figure}[!htb]
    \centering
   \includegraphics[width=0.42\textwidth]{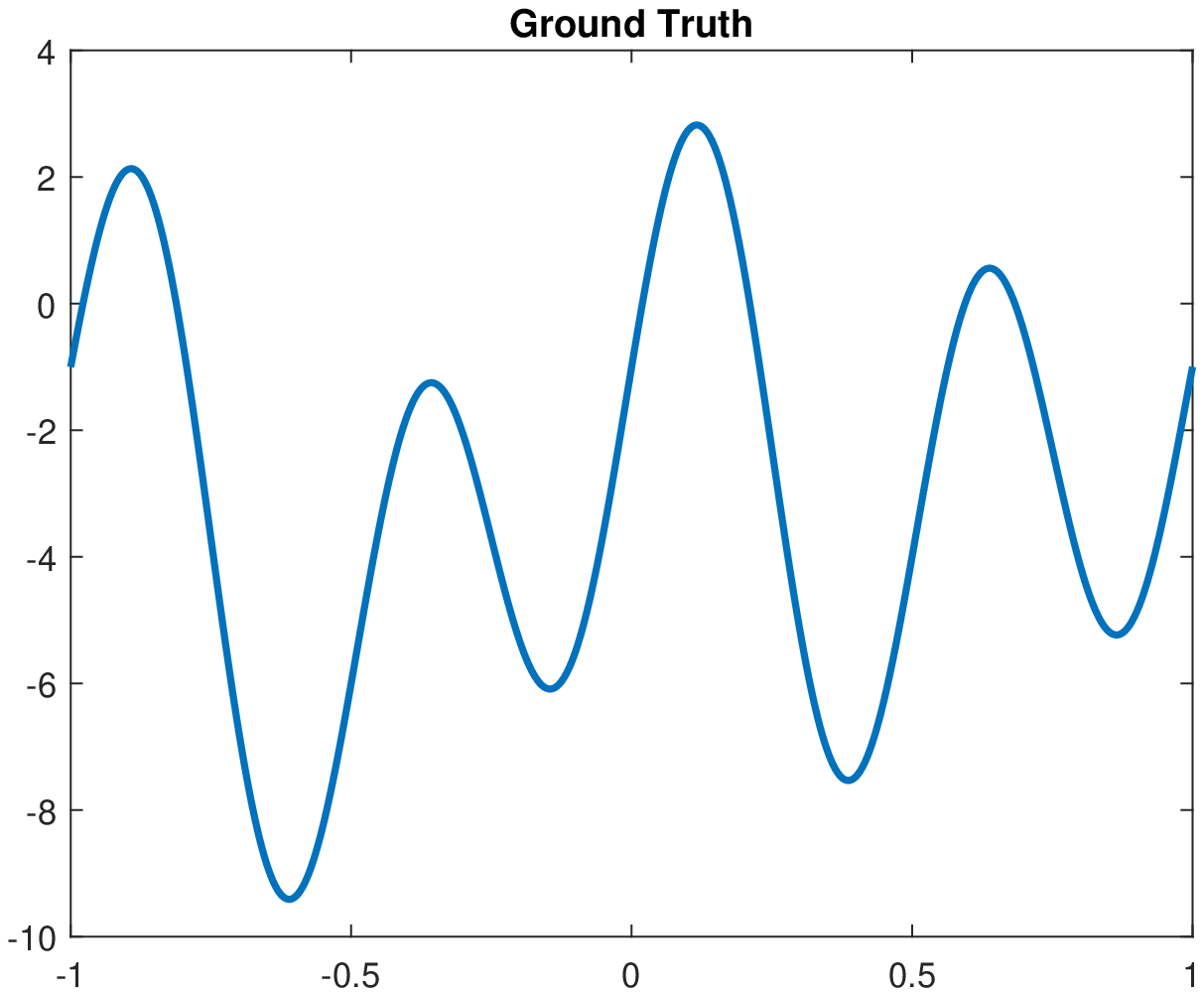}
    \caption{
    Ground truth $\dot{q}=\sin(\pi x)+2\cos(2\pi x)+4\sin(4\pi x)-3$. 
    }
    \label{fig:c1}
\bigskip
    \includegraphics[width=0.42\textwidth]{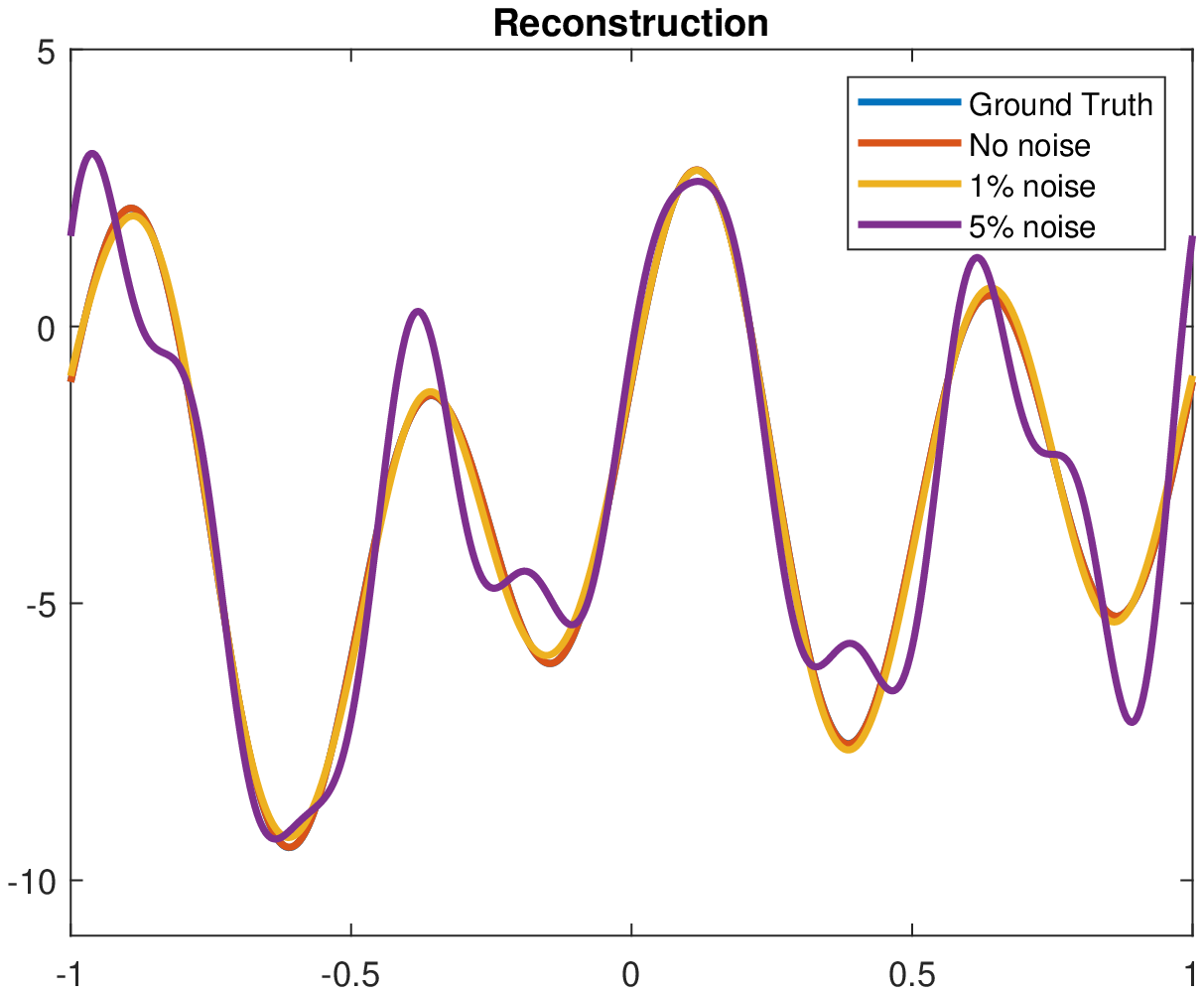}
    \includegraphics[width=0.42\textwidth]{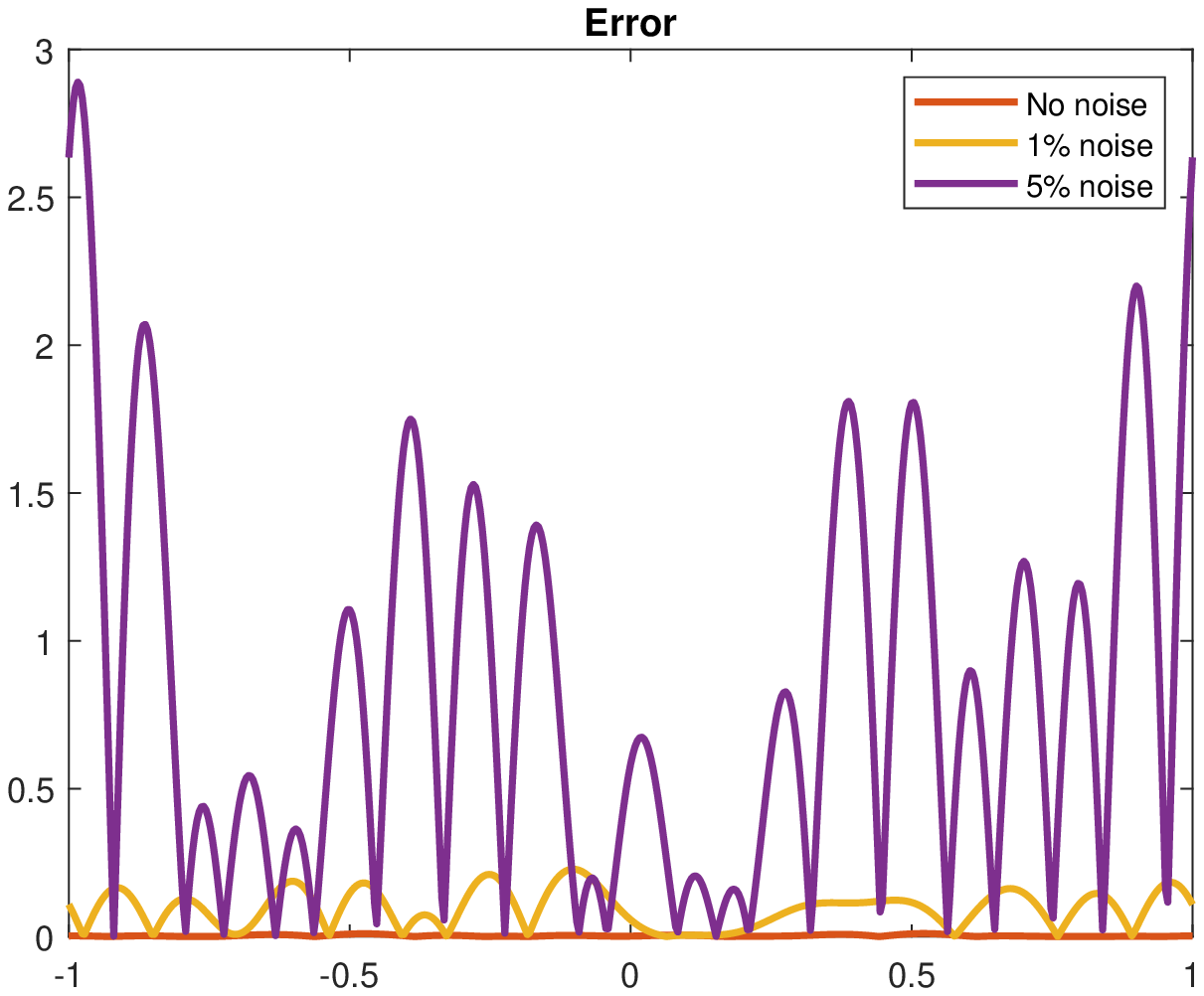}
    \caption{Left: Reconstructed $\dot{q}$ with $0\%,1\%,5\%$ Gaussian noise and the ground truth. Right: The corresponding error between the reconstruction and the ground truth. The relative $L^2$-errors are $0.1\%$, $2.5\%$, and $23.9\%$ 
    respectively}
    \label{fig:exp1}
\bigskip
    \includegraphics[width=0.42\textwidth]{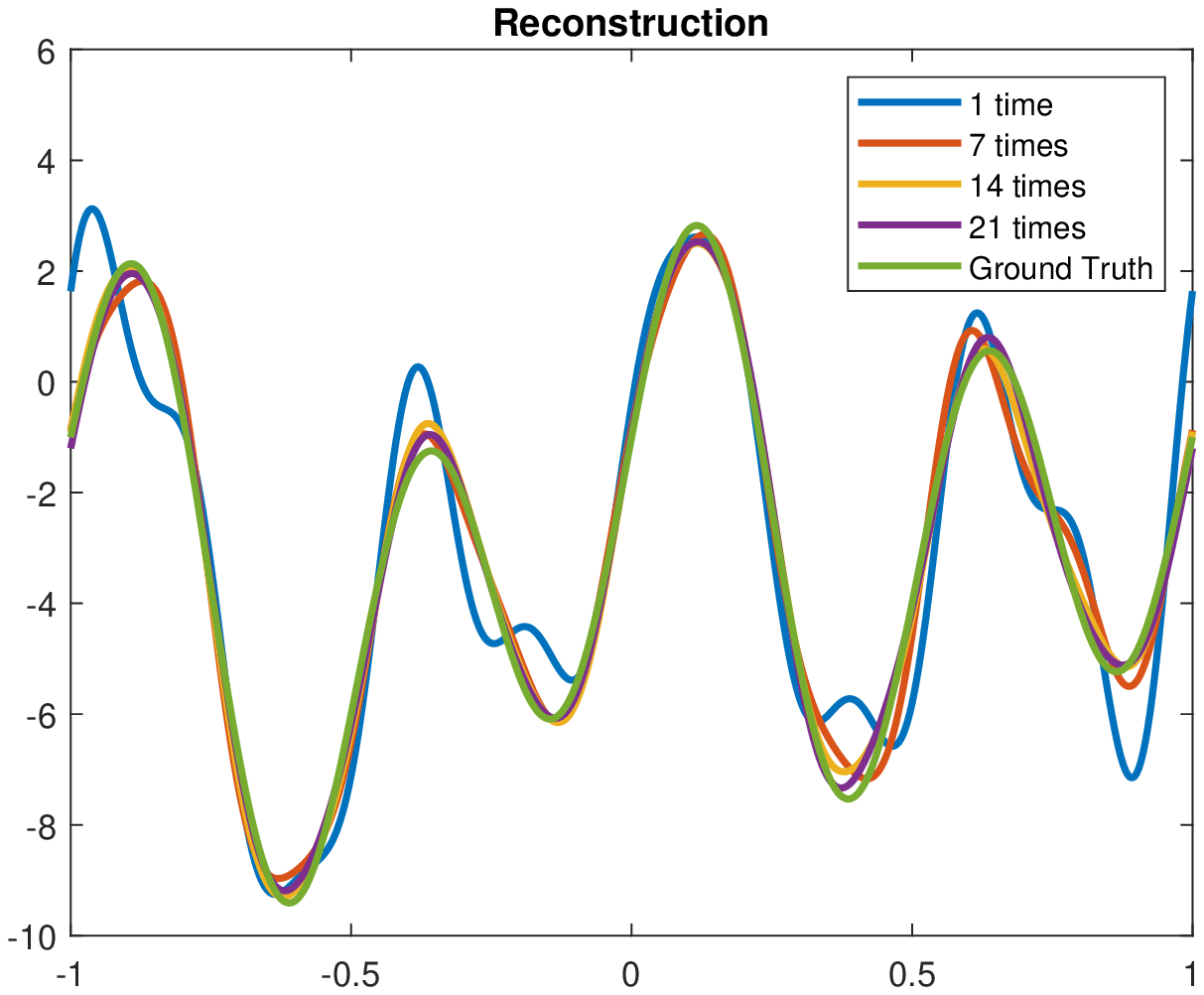}
    \includegraphics[width=0.42\textwidth]{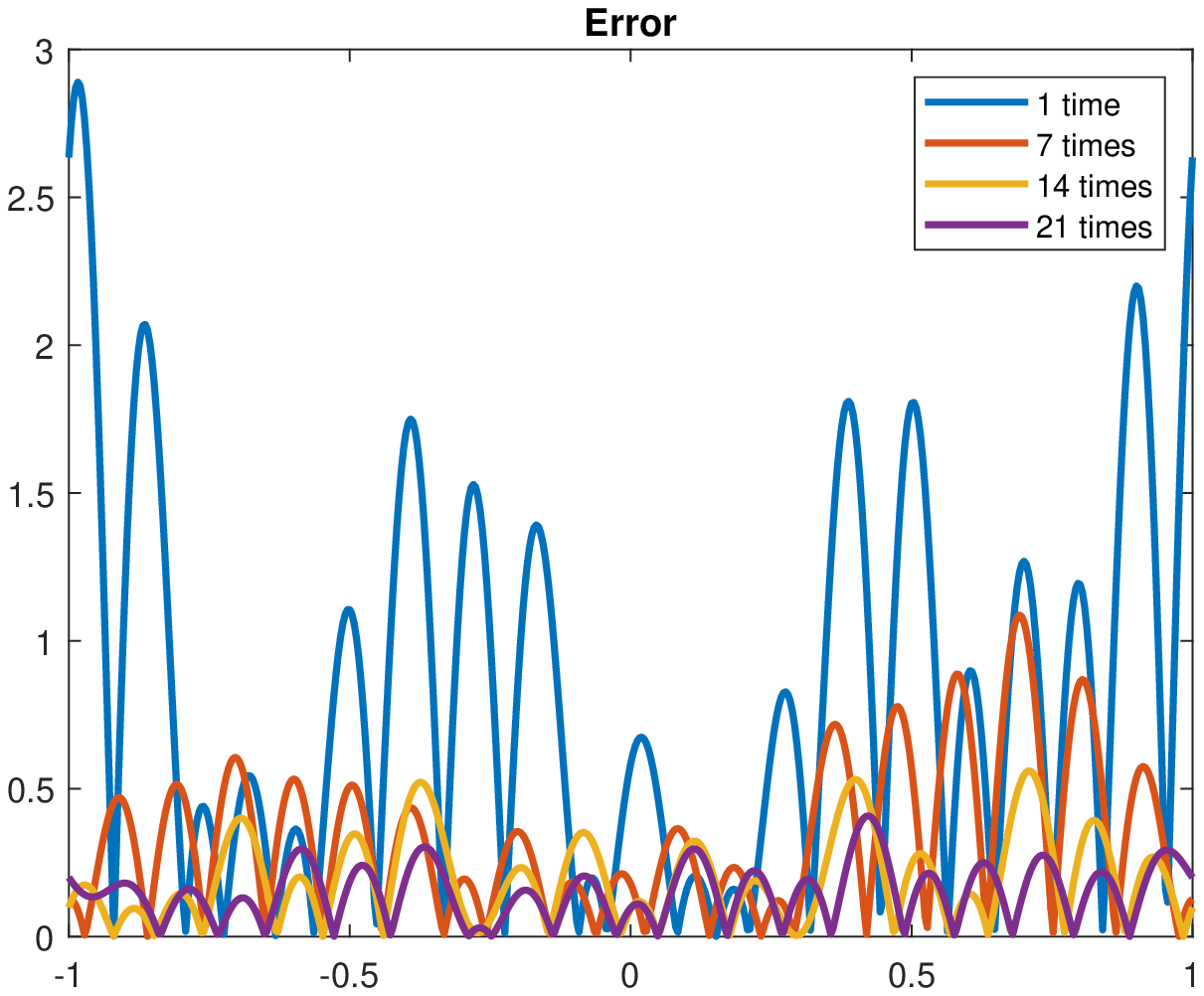}
    \caption{Left: Reconstructed $\dot{q}$ under $1,7,14,21$ times repetition with $5\%$ Gaussian noise and the ground truth. Right: The corresponding error between the reconstruction and the ground truth. 
    The relative $L^2$-errors are $23.9\%$, $9.4\%$, and $5.5\%$, respectively.
    }
    \label{fig:repeat1}
\end{figure}


\bigskip
\textbf{Experiment 2.}
The second experiment tests reconstruction of a discontinuous $\dot{q}$. we choose $\dot{q}$ to be the Heaviside function
$$
    H(x)=\begin{cases}
    1&x\geq0,\\
    0&x<0.
    \end{cases}
$$
The Fourier series of $H(x)$ on $\Omega=[-1,1]$ is
$$
    H(x)=\frac{1}{2}+\sum_{n=1}^\infty\frac{2}{(2n-1)\pi}\sin((2n-1)\pi x).
$$
With the choice of the finite computational basis, we can only expect to reconstruct the following orthogonal projection:
$$
    H_{N}(x):=\frac{1}{2}+\sum_{n=1}^{\left\lceil\frac{N}{2}\right\rceil}\frac{2}{(2n-1)\pi}\sin((2n-1)\pi x),
$$
see Figure~\ref{fig:c2} for the graph of $H(x)$ and $H_N(x)$. The reconstruction formula~\eqref{eq:reconformula} is implemented with $0\%,1\%$, and $5\%$ of Gaussian noise added to $\dot{\Lambda}_{\dot{q}}$, respectively. The reconstructions and corresponding errors with a single measurement are illustrated in Figure~\ref{fig:exp2}. The averaged reconstruction with $5\%$ of Gaussian noise and multiple repeated measurements are illustrated in Figure~\ref{fig:repeat2}.

\begin{figure}[!htb]
    \centering
    \includegraphics[width=0.42\textwidth]{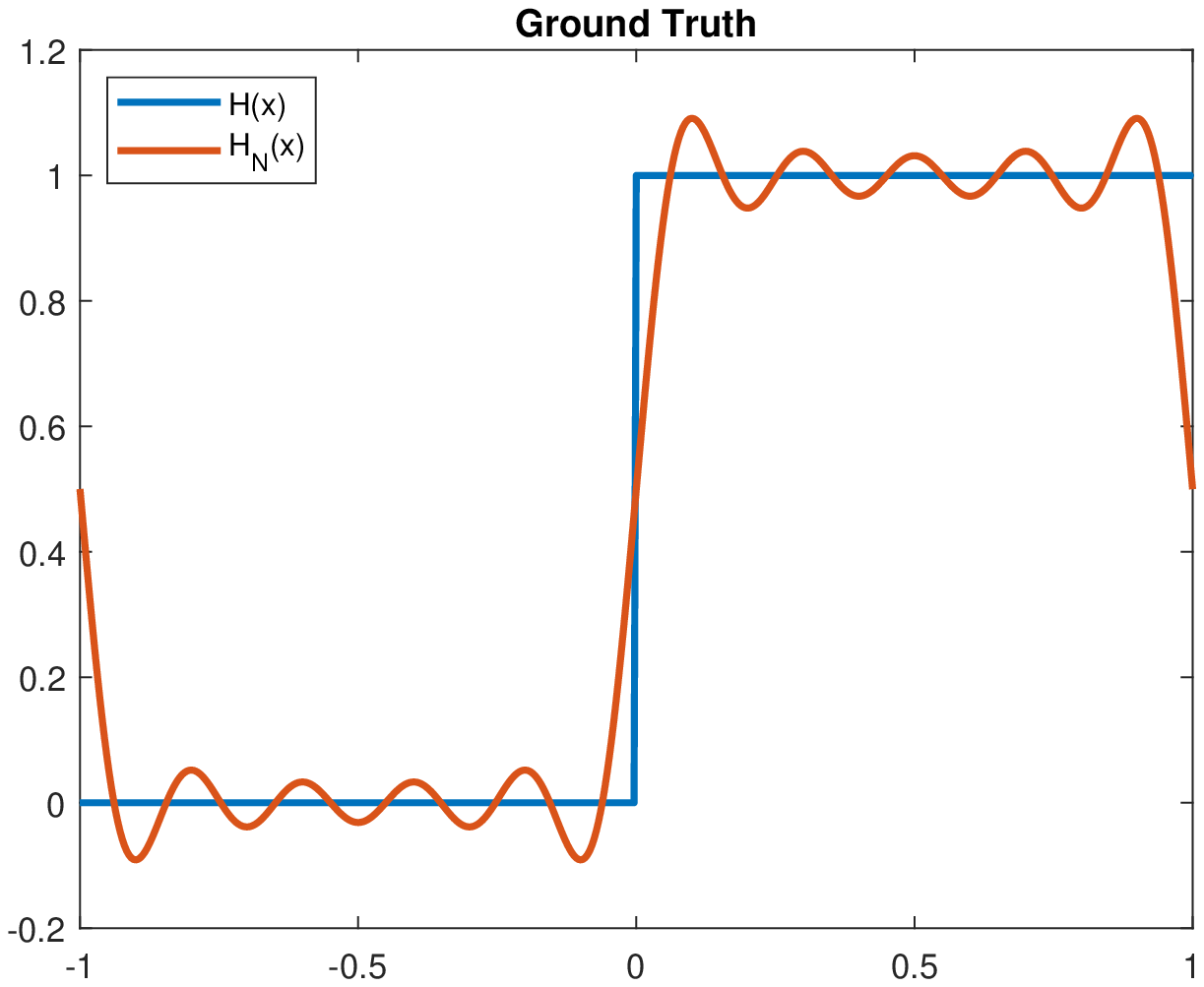}
    \caption{
    Ground Truth $\dot{q}=H(x)$ and its projection $H_N(x)$.
    }
    \label{fig:c2}
\bigskip
    \includegraphics[width=0.42\textwidth]{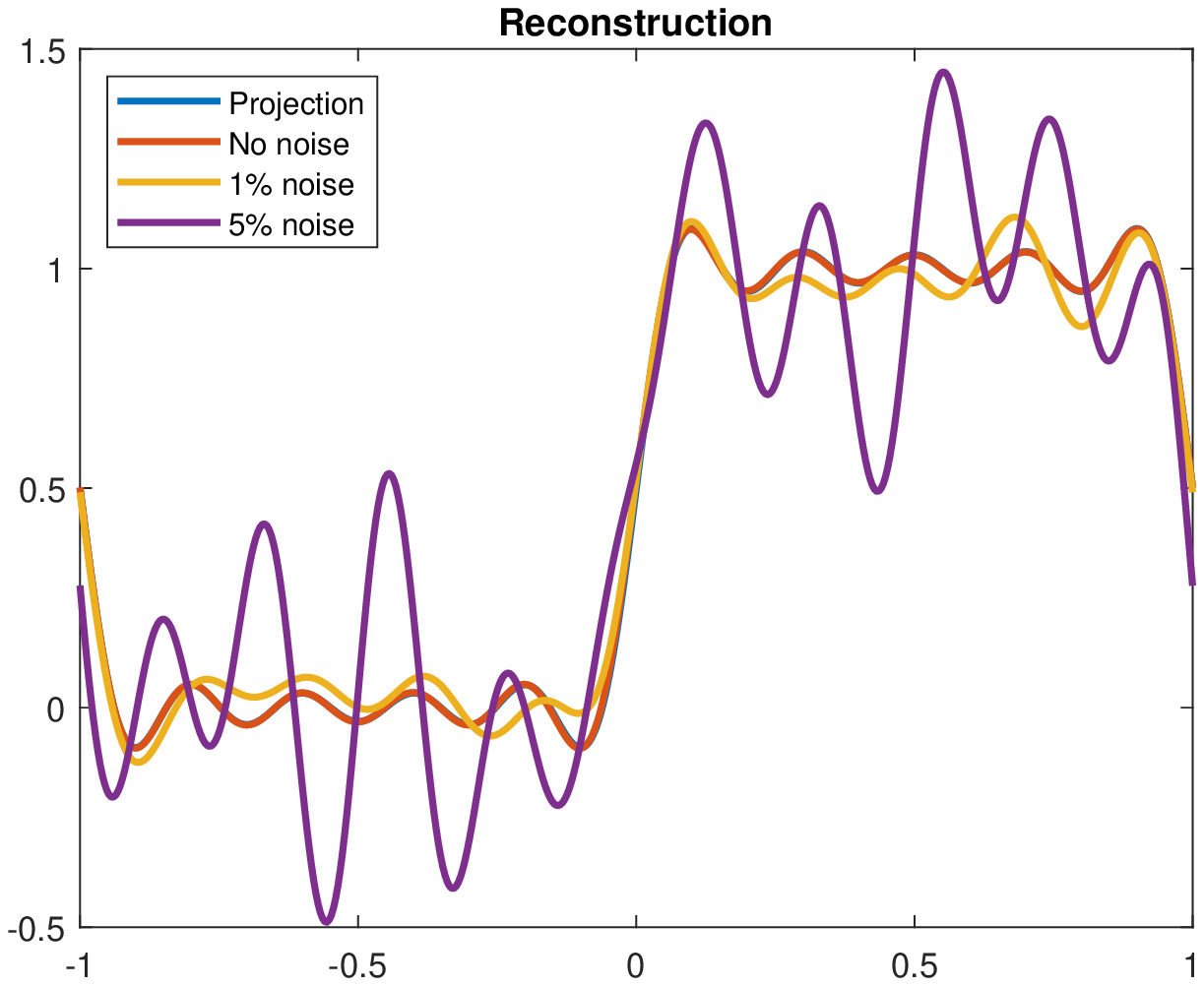}
    \includegraphics[width=0.42\textwidth]{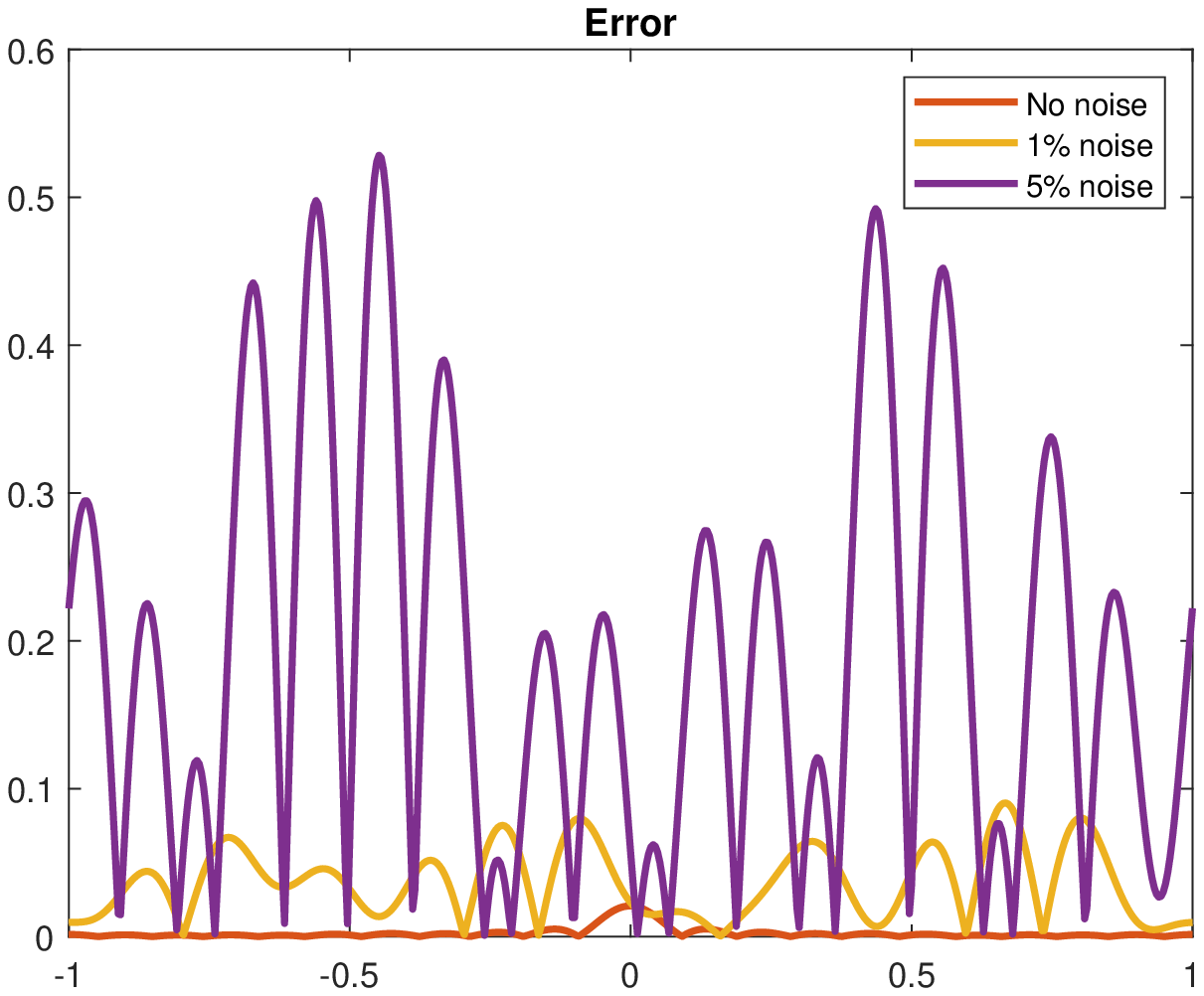}
    \caption{Left: Reconstructed $\dot{q}$ with $0\%,1\%,5\%$ Gaussian noise and the projection of the ground truth. Right: The corresponding error between the reconstruction and the projection of the ground truth. 
    The relative $L^2$-errors are $0.6\%$, $6.2\%$, and $33.8\%$, respectively.
    }
    \label{fig:exp2}
\bigskip
    \includegraphics[width=0.42\textwidth]{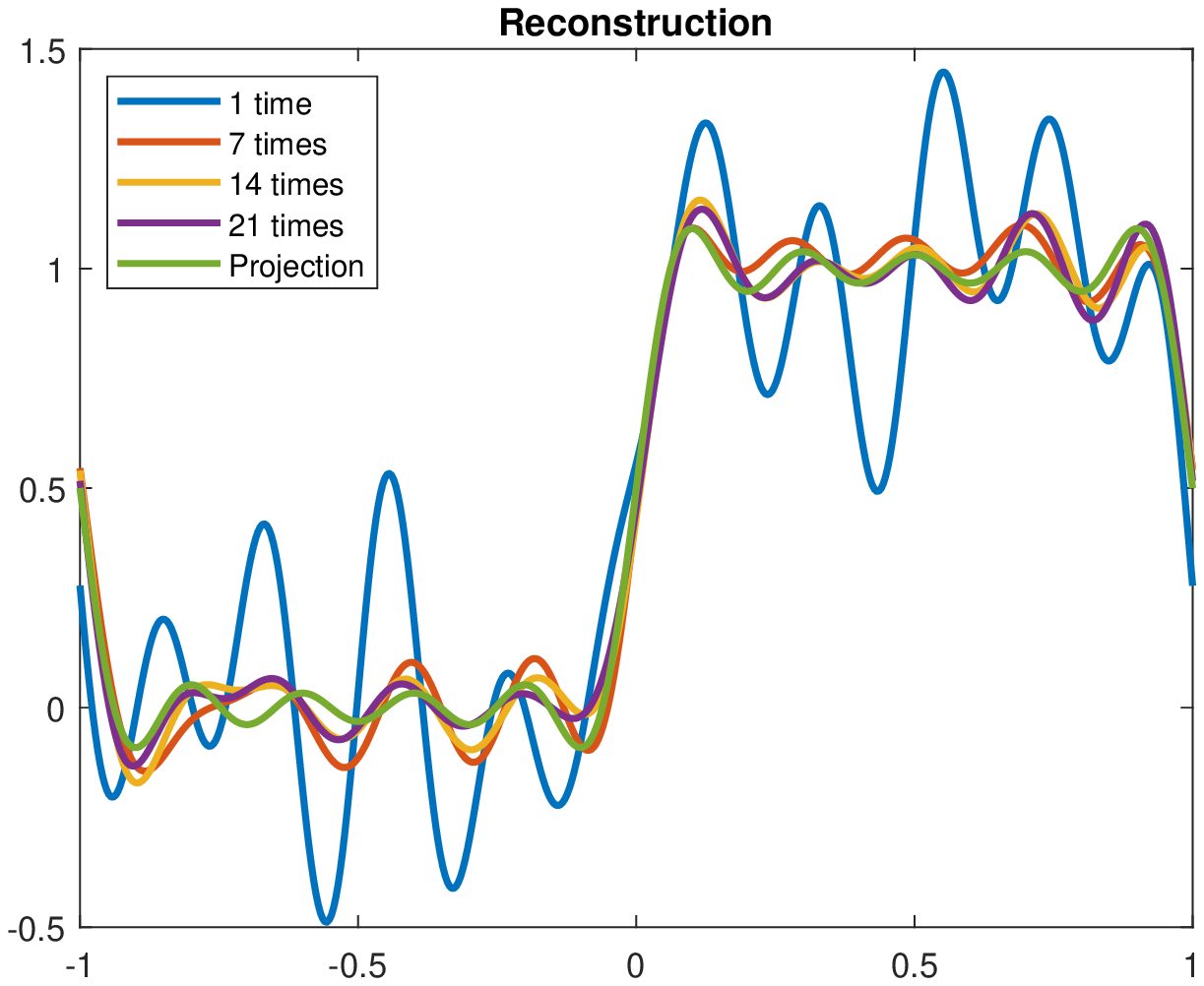}
    \includegraphics[width=0.42\textwidth]{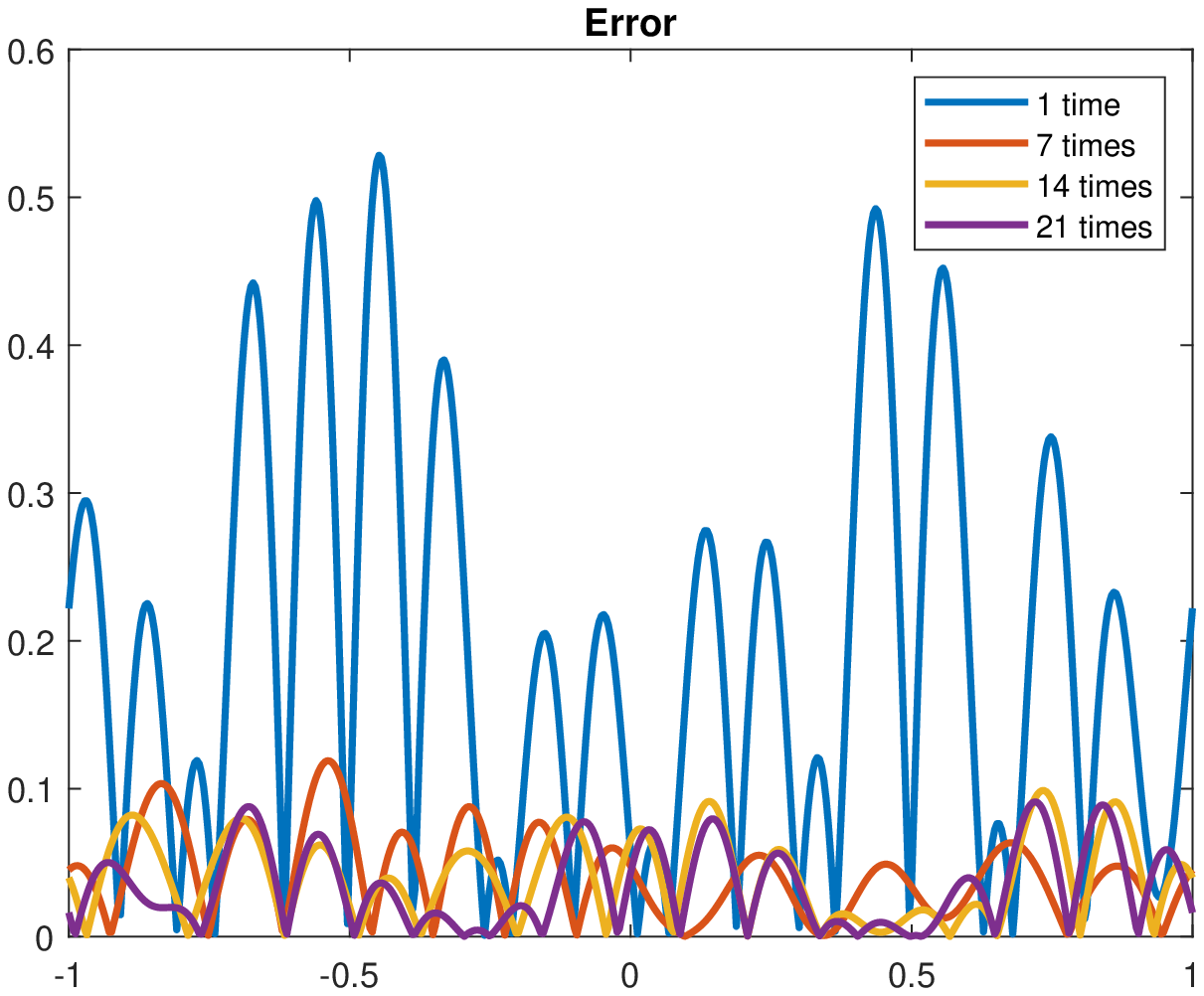}
    \caption{Left: Reconstructed $\dot{q}$ under $1,7,14,21$ times repetition with $5\%$ Gaussian noise and the projection of the ground truth. Right: The corresponding error between the reconstruction and the projection of the ground truth. 
    The relative $L^2$-errors are $33.8\%$, $7.2\%$, and $6.9\%$, respectively.
    }
    \label{fig:repeat2}
\end{figure}


\bigskip
\textbf{Experiment 3.}
In this experiment, we apply the reconstruction formula~\eqref{eq:reconformula} to measurement from the non-linear IBVP. Specifically, we choose $c_0= 1$ and the potential
\[q=q_0+\varepsilon\dot{q}+\varepsilon^2 \ddot{q}\]
where $\epsilon>0$ is a small number, the background potential $q_0= 0$, and the perturbations
\[\dot{q}=\sin(\pi x)+2\cos(2\pi x)+4\sin(4\pi x)-3,\qquad\ddot{q}=20\cos(20\pi x).\]
We apply the reconstruction formula~\eqref{eq:reconformula}, with $\dot{\Lambda}_{\dot{q}} f$ replaced by $\Lambda_q f - \Lambda_{q_0} f$ (where $\Lambda_q f$ and $\Lambda_{q_0} f$ are computed by solving the boundary value problem~\eqref{eq:bvp}), and then add $q_0$ to the reconstruction to obtain an approximation of $q$. 
The rationale is that, when $\epsilon$ is small, we have the following approximation
\[\Lambda_q-\Lambda_{q_0}\approx\varepsilon\dot{\Lambda}_{\dot{q}} = \dot{\Lambda}_{\epsilon\dot{q}}.\]
Therefore, the resulting reconstruction from~\eqref{eq:reconformula} is approximately $\epsilon\dot{q}$, and adding $q_0$ to it yields a linear approximation of $q$. 
We choose $\epsilon=0.1$. The ground truth is illustrated in Figure~\ref{fig:c5}.

Adding noise is a bit delicate in this experiment. The noise to be added is additive and proportional to the magnitude of the signal.
We tested two ways of adding noise: (1) adding noise to the difference $(\Lambda_q - \Lambda_{q_0})f$; (2) adding noises to $\Lambda_q f$ and $\Lambda_{q_0} f$ respectively, then subtract to find the difference. The resulting numerical performance are different, and it turns out the former introduces much less error than the latter.
This is because the numerical values in the discretization of $(\Lambda_q - \Lambda_{q_0})f$ are much smaller, hence the proportional noise is relatively small. In contrast, the numerical values in the discretization of $\Lambda_q f$ and $\Lambda_{q_0}f$ are larger, hence the proportional noise is relatively large. Comparison of the reconstruction errors are shown in Figure~\ref{fig:exp5} and Figure~\ref{fig:exp6}.
Note that the noise added in Figure~\ref{fig:exp6} (0.1\% and 0.5\%) is only a tenth of that in Figure~\ref{fig:exp5} (1\% and 5\%).

\begin{figure}[!htb]
    \centering
    \includegraphics[width=0.42\textwidth]{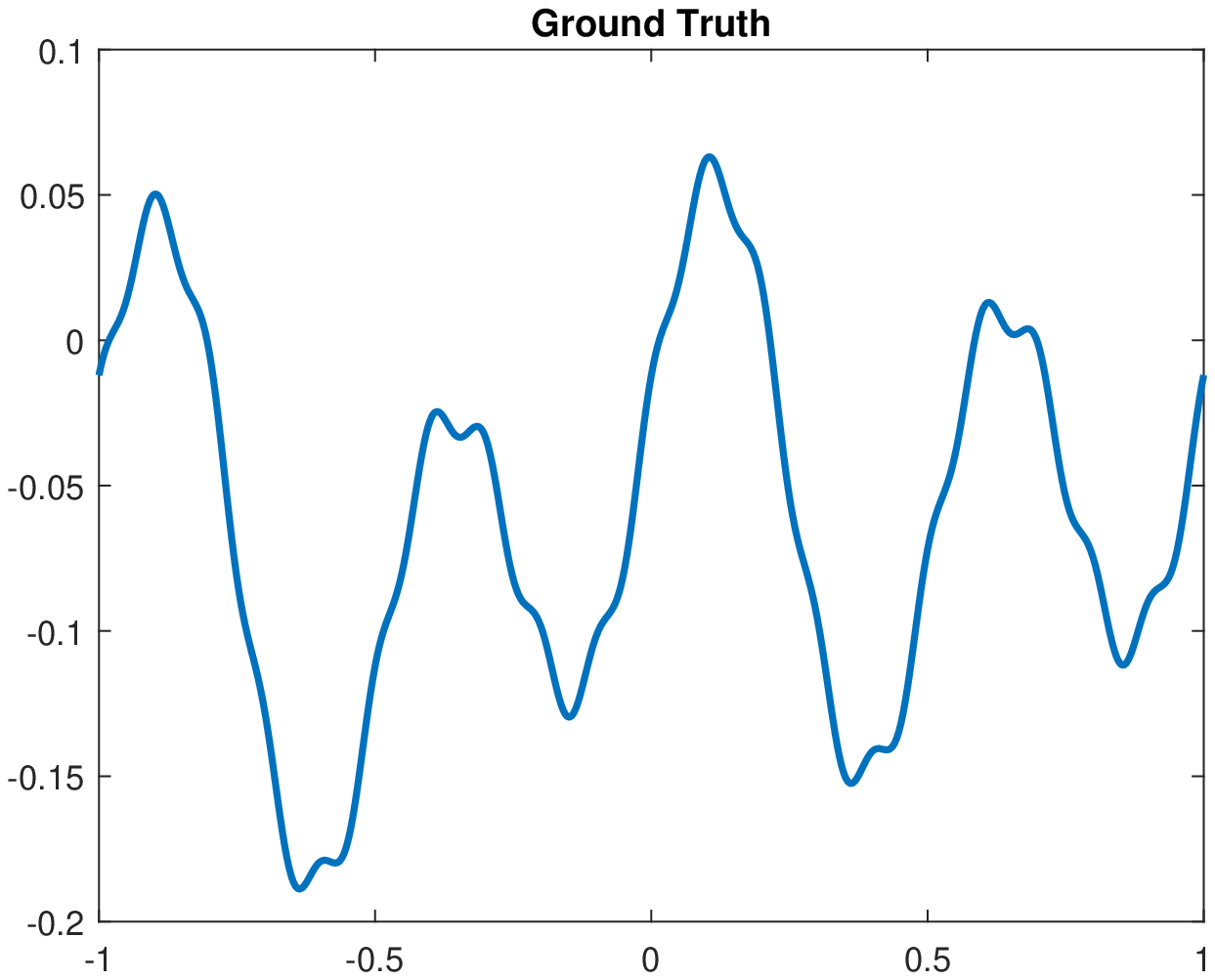}
    \caption{Ground truth $q=q_0+\varepsilon\dot{q}+\varepsilon^2\ddot{q}$, where $\dot{q}=\sin(\pi x)+2\cos(2\pi x)+4\sin(4\pi x)-3$, $\ddot{q}=20\cos(20\pi x)$, $\varepsilon=0.02$.}
    \label{fig:c5}
\bigskip
    \includegraphics[width=0.42\textwidth]{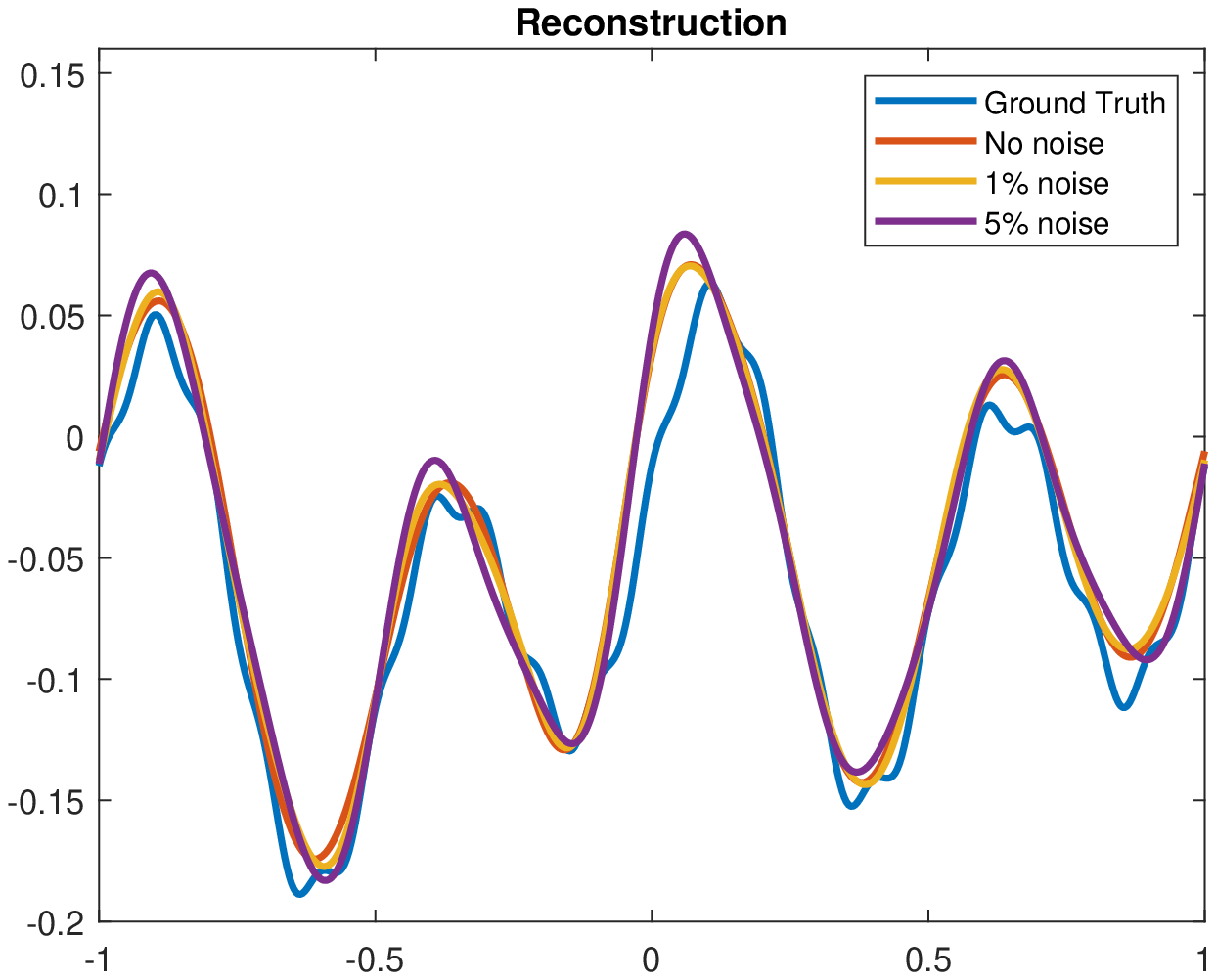}
    \includegraphics[width=0.42\textwidth]{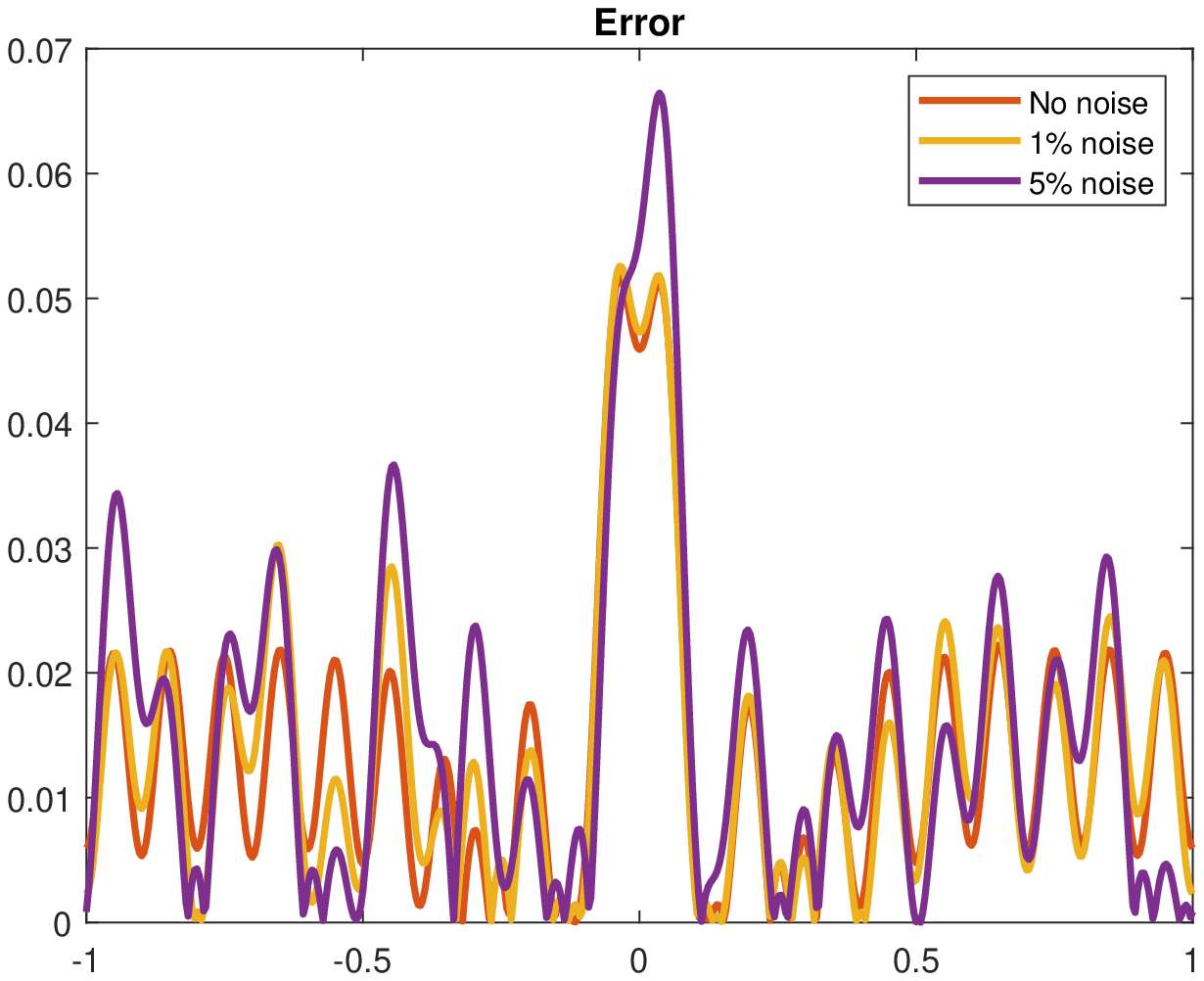}
    \caption{Left: Reconstructed $q$ with $0\%,1\%,5\%$ Gaussian noise added to $\Lambda_q-\Lambda_{q_0}$ and the ground truth. Right: The corresponding error between the reconstruction and the ground truth. 
    The relative $L^2$-errors are $19.9\%$, $20.4\%$, and $23.5\%$, respectively.
    }
    \label{fig:exp5}
\bigskip
    \includegraphics[width=0.42\textwidth]{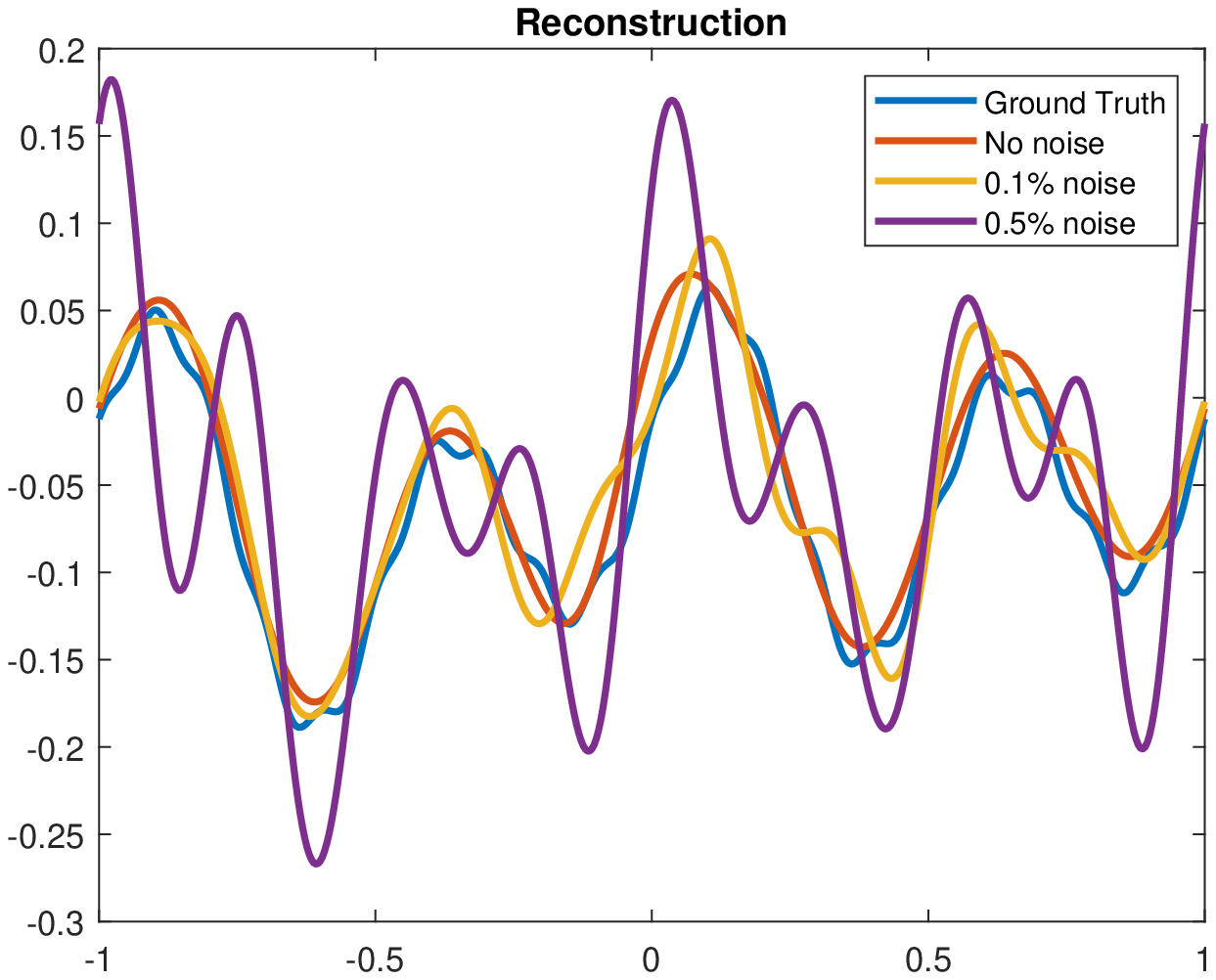}
    \includegraphics[width=0.42\textwidth]{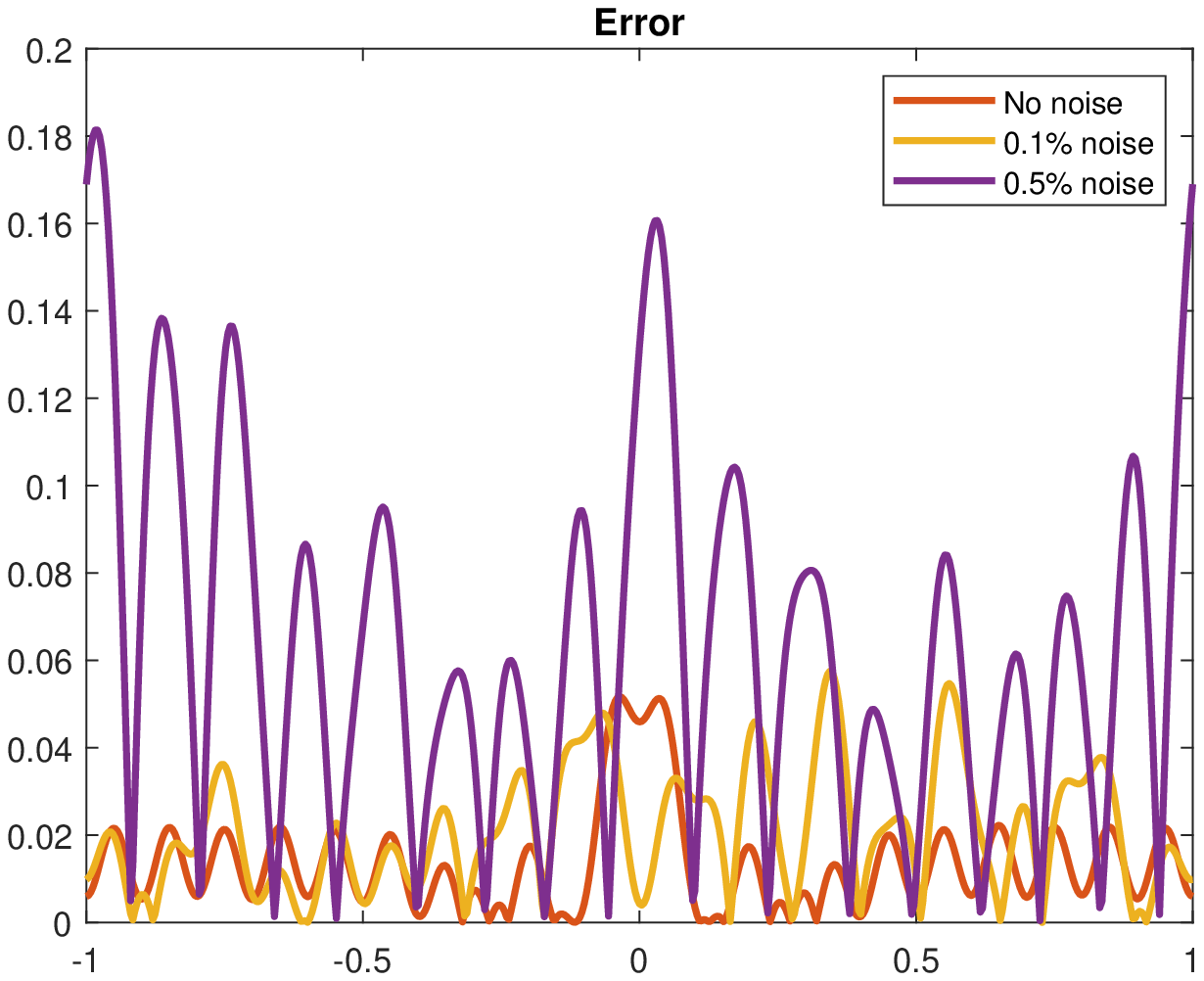}
    \caption{Left: Reconstructed $q$ with $0\%,0.1\%,0.5\%$ Gaussian noise added to $\Lambda_q$ and $\Lambda_{q_0}$ independently, and the ground truth. Right: The corresponding error between the reconstruction and the ground truth. 
    The relative $L^2$-errors are $19.9\%$, $28.1\%$, and $86.2\%$, respectively.
    }
    \label{fig:exp6}
\end{figure}

\section*{Acknowledgment}
The research of T. Yang and Y. Yang is partially supported by the NSF grant DMS-1715178, DMS-2006881, and startup fund from Michigan State University.


\appendix

\section*{Appendix}

In this appendix, we collect a few results that are used in the main text. First, we present a lemma regarding uniform continuity of min/max functions, which is used in the proof of Proposition~\ref{thm:control}.

\begin{lemma} \label{thm:unicont}
Let $X,Y$ be compact metric spaces and suppose that $f : X \times Y \to \R$ is uniformly continuous. Then the function $F: X \to \R$
    \begin{align*}
F(x) := \min \{f(x,y) : y \in Y\}
    \end{align*}
is uniformly continuous. The same is true if $\min$ is replaced by $\max$.
\end{lemma} 
\begin{proof}
Let $\epsilon > 0$ and let $x,x' \in X$.
We may assume without loss of generality that $F(x) \ge F(x')$.
Let $y,y' \in Y$ be such that $F(x)=f(x,y)$ and $F(x')=f(x',y')$, then
    \begin{align*}
|F(x) - F(x')| = f(x,y) - f(x',y') \le f(x,y') - f(x',y') 
    \end{align*}
Due to uniform continuity there is $\delta > 0$ such that 
    \begin{align*}
|f(x,y') - f(x',y')| < \epsilon
    \end{align*}
whenever $d(x,x') < \delta $. The proof when ``$\min$'' is replaced by ``$\max$'' is similar.
\end{proof}

\bigskip
Next, we construct a Helmholtz solution of the form~\eqref{eq:phi}, with $r$ satisfying the asymptotic condition~\eqref{eq:rdecay}.

\begin{lemma} \label{thm:resolest}
Let $\lambda \geq 0$. The perturbed Helmholtz equation
$(\Delta + \lambda - q_0)\phi=0$ admits solutions of the form
$$
\phi(x) = e^{i\sqrt{\lambda} \theta\cdot x} + r(x;\lambda)
$$
for any $\theta\in\mathbb{S}^{n-1}$, with $\|r\|_{H^s(\mathbb{R}^n)} = O(\lambda^\frac{s-1}{2})$ as $\lambda\rightarrow\infty$ for any $s\geq 0$.
\end{lemma}

\begin{proof}
If $\phi(x) = e^{i\sqrt{\lambda} \theta\cdot x} + r(x;\lambda)$ solves the perturbed Helmholtz equation and $\theta\in\mathbb{S}^{n-1}$, then $r$ has to satisfy
\begin{equation} \label{eq:scatter}
(\Delta + \lambda - q_0) r =  q_0 e^{i\sqrt{\lambda} \theta\cdot x} \quad\quad \text{ in } \Omega.
\end{equation}
Such $r$ can be constructed as follows. Let $R(\lambda):= (\Delta + \lambda - q_0)^{-1}$ be the outgoing resolvent of the perturbed Helmholtz operator. For any compactly supported smooth function $\chi\in C^\infty_c(\mathbb{R}^n)$, the following $L^2$-resolvent estimate holds~\cite[Theorem 3.1]{dyatlov2019mathematical}:
\begin{equation}
\|\chi R(\lambda) \chi\|_{L^2(\mathbb{R}^n)\rightarrow L^2(\mathbb{R}^n)} \leq \frac{C}{\sqrt{\lambda}}
\end{equation}
where $C>0$ is a constant independent of $\lambda$.
In particular, we choose $\chi\in C^\infty_c(\mathbb{R}^n)$ such that $\chi=1$ on $\overline{\Omega}$ and set $r:=\chi R(\lambda) \chi(q_0 e^{i\sqrt{\lambda} \theta\cdot x})$, then $r$ satisfies~\eqref{eq:scatter}, $r$ is compactly supported, and
\begin{equation}
 \|r\|_{L^2(\mathbb{R}^n)} \leq \frac{C}{\sqrt{\lambda}}.
\end{equation}
This proves the claim when $s=0$.

We proceed by induction. Suppose it has been proved that $r$ is compactly supported and satisfies $\|r\|_{H^s(\mathbb{R}^n)} = O(\lambda^\frac{s-1}{2})$ as $\lambda\rightarrow\infty$ for some integer $s\geq 0$. Let $\Omega_1$, $\Omega_2$ be two open sets such that $\Omega\subset \supp\chi \subset\Omega_1 \subset \overline{\Omega_1} \subset \Omega_2$. We apply the interior regularity estimate~\cite[Section 6.3, Theorem 2 ]{evans1998partial} to the equation $(\Delta - q_0) r =  q_0 e^{i\sqrt{\lambda} \theta\cdot x} -\lambda r$ to obtain
\begin{align*}
\|r\|_{H^{s+2}(\mathbb{R}^n)} = 
\|r\|_{H^{s+2}(\Omega_1)} 
& \leq C \left(
\|q_0 e^{i\sqrt{\lambda} \theta\cdot x} -\lambda r\|_{H^s(\Omega_2)} + \|r\|_{H^s(\Omega_2)} \right) \\
& \leq C \left(
1 + \lambda \|r\|_{H^s(\mathbb{R}^n)} + \|r\|_{H^s(\mathbb{R}^n)} \right) \leq C (1+\lambda^{\frac{s+1}{2}}).
\end{align*}
This completes the inductive step, hence the claim holds for all even integers $s$. The general claim is a consequence of interpolation.
\end{proof}





\bibliographystyle{abbrv}
\bibliography{refs}

\end{document}